\documentclass[11pt,reqno]{article}
\usepackage{amsmath}
\usepackage{mathrsfs}
\usepackage{amsfonts}
\usepackage{amssymb}

\usepackage{amssymb}
\usepackage{amsthm}
\usepackage{graphicx}              
\usepackage{amsmath}               
\usepackage{amsfonts}              
\usepackage{amsthm}                
\usepackage{setspace}
\usepackage{epstopdf}
\usepackage{epsfig}
\usepackage{fancyhdr}
\newtheorem{theorem}{Theorem}[section]
\newtheorem{proposition}{Proposition}[section]
\newtheorem{lemma}{Lemma}[section]
\newtheorem{corollary}{Corollary}[section]
\newtheorem{remark}{Remark}[section]

\renewcommand{\thefootnote}{\fnsymbol{footnote}}
\makeatletter

\newcommand{\Rmnum}[1]{\expandafter\@slowromancap\romannumeral #1@}
\makeatother

\allowdisplaybreaks
\numberwithin{equation}{section}

\title{Endpoint Strichartz estimates for magnetic wave equations on $\Bbb H^2$}
\author{
{Ze Li}\\
{{Academy of Mathematics and Systems Science (AMSS)}}\\
{{Chinese Academy of Sciences (CAS)}}\\
{{Beijing 100190, P. R. China}}}
\begin{document}
\maketitle
\medskip
\date
\pagestyle{fancy}
\lhead{ }
\chead{Endpoint Strichartz Estimates for Magnetic Wave Equations on $\Bbb H^2$}
\rhead{ }

\renewcommand{\thefootnote}{\fnsymbol{footnote}}
\footnotetext{E-mail: rikudosennin@163.com  }

\noindent{\bf Abstract}
In this paper, we prove that Kato smoothing effects for magnetic half wave operators can yield the endpoint Strichartz estimates for linear wave equation with magnetic potential on two dimensional hyperbolic spaces.
This result serves as a cornerstone for the author's work \cite{Laiazaea} and collaborative work \cite{LaMaZa} in the study of asymptotic stability of harmonic maps for wave maps from $\Bbb R\times\Bbb H^2$ to $\Bbb H^2$.

\bigskip

\section{Introduction}
The motivation of this problem is the study of asymptotic stability of harmonic maps for wave maps from $\Bbb R\times\Bbb H^2$ to $\Bbb H^2$. In fact, the heat tension field $\phi_s$ which provides a natural measure for the distance between the solution of wave maps and the limit harmonic map   satisfies a master semilinear wave equation under Tao's caloric gauge. After separating the limit part of connections and differential fields and applying ``dynamic separation", the linear part of the master equation becomes a wave equation with magnetic potential:
\begin{align}\label{1}
\left\{ \begin{array}{l}
\Box u+B(x)u+2(A,{du})-(d^*A)u=F, \\
u(0,x) = {u_0}:\Bbb H^2\to \Bbb R^2,{\partial _t}u(0,x) = {u_1}(x): \Bbb H^2\to \Bbb R^2\\
\end{array} \right.
\end{align}
where $u$ is a $\Bbb R^2$-valued field defined on $\Bbb R\times \Bbb H^2$, $A=A_idx^i$ is a real anti-symmetric $2\times2$ matrix valued one form defined on $\Bbb H^2$, $B$ is a real symmetric $2\times2$ matrix defined on $\Bbb H^2$.  $(A,{du})$ denotes the metric of one forms. Since $A$ is a matrix valued one form and $du$ is a $\Bbb R^2$ valued one form, $(A,{du})$ is $\Bbb R^2$ valued as well. And $\Box=-\partial_t^2+\Delta_{\Bbb H^2}$ is the D'Alembertian on $\Bbb R\times \Bbb H^2$. Integration by parts shows $Wu:=B_0(x)u+2(A,{du})-(d^*A)u$ is formally symmetric in $L^2(\Bbb H^2;\Bbb C^2)$. In \cite{Laiazaea}, $A=A_idx^i$ is indeed the connection one form on the pullback bundle $Q^*(T\Bbb H^2)$ with $Q:\Bbb H^2\to\Bbb H^2$ denoting the limit harmonic map.  As in the Euclidean case, we may use the tight notation $Wu=Vu+Xu$ to denote the potential part, where $V:=B(x)-(d^*A) $ is a matrix valued function defined on $\Bbb H^2$ and $X:=2h^{ij}A_i\frac{\partial}{\partial x^j}$ is a matrix valued vector field defined on $\Bbb H^2$.

The wave map equation on flat spacetimes known as the nonlinear $\sigma$-model, arises as a model problem in general relativity and particle physics, see for instance \cite{MaSa,AaGaSa}. The Cauchy and dynamic problems for wave maps on flat spacetimes have been a fruitful field with plenty of works, see for instance \cite{KaSa,SaTa1a,Taataaaraua3a,Taataaaraua4a,Taaaoa1a,Taaaoa2a}.
The dynamics for the wave map equations on curved spacetimes were less understood until now. We mention the work of Shatah, Tahvildar-Zadeh \cite{SaTaZa1a} on the $\Bbb S^2\times\Bbb R$ background and Lawrie, Oh, Shahshahani \cite{LaOaSa2a,LaOaSa,LaOaSa4a,LaOaSa3a} on the $\Bbb H^n\times \Bbb R$ background.

In this paper, we focus on the endpoint Strichartz estimates for the magnetic wave equation (\ref{1}). The Strichartz estimates for magnetic Schr\"odinger equations (MS), magnetic Dirac equations (MD) and magnetic wave equations (MW) on flat spaces were intensively studied in decades, for instance \cite{Ya,AaHaSa,EaGaSa1a,Ga,DG2,DaGa3a,Da1a}. In the fundamental work of Rodnianski, Schalg \cite{RaSa}, they showed that the Kato smoothing effects imply the non-endpoint Strichartz estimates for MS.  This idea was further developed to MS with large potential by Erdogan, Goldberg, Schlag \cite{EaGaSa1a} and MW, MD by D'Ancona, Fanelli \cite{DG2}. The endpoint Strichartz estimates for free wave and Schr\"odinger equations were obtained first by Keel, Tao \cite{KaTa}. With a key lemma of Ionescu, Kenig \cite{IaKa} whose proof is based on \cite{KaTa}, D'Ancona, Fanelli, Vega, Visciglia \cite{DaGa3a} obtained endpoint Strichartz estimates for MS in the small potential case. Strichartz estimates for free Schr\"odinger, wave and Klein Gordon equations on $\Bbb H^n$ were obtained by Tataru \cite{Taataaaraua4a}, Metcalfe, Taylor \cite{MaTa1a,MaTa2a}, Anker, Pierfelice, Vallarino \cite{AaPa,AaPa2a} and see Metcalfe, Tataru \cite{MaTa3a} for small perturbations of flat spacetimes. And the study of resolvent estimates, spectral measures, scattering, analytic continuation, degenerate elliptic operators, etc. on hyperbolic/asymptotic hyperbolic spaces has become an active field, see the works \cite{MBV,Vasy1,Vasy2,CH} for instance. The dispersive estimates of Schr\"odinger operators with electric potential on $\Bbb H^d$ were obtained by Borthwick, Marzuola \cite{BaMa} for $t\ge1$.

Our main theorems consist of two parts. The first result shows that the Kato smoothing effect estimates for magnetic half wave operators on $\Bbb H^2$ imply both the non-endpoint Strichartz and endpoint Strichartz estimates. Second, we prove the Kato smoothing effect estimates in the small potential case. Thus, by our first result the endpoint Strichartz estimates hold in the small potential case, which is useful for \cite{LaMaZa}. And we remark that for the special magnetic Schr\"odinger operator appearing in the study of wave maps from $\Bbb R\times\Bbb H^2$ to $\Bbb H^2$, the Kato smoothing effect can also be established for arbitrary large potentials, see \cite{Laiazaea}.

Let $\Bbb D$ denote the Poincare disk model for $\Bbb H^2$.
Let $r=d(x,O)$ be the geodesic distance between $x\in\Bbb D$ and the origin point $O$ in $\Bbb D$. Recall $V:=B(x)-d^*A$.
Our main theorems are as follows.
\begin{theorem}\label{2}
Suppose that $B,A$ satisfy for some $\varrho>0$
\begin{align}\label{Vab}
\|Ve^{r\varrho}\|_{L^{\infty}_x}+\|e^{r\varrho}|A|\|_{L^{\infty}_x}<\infty,
\end{align}
and the Schr\"odinger operator $H=-\Delta+V+X$ is strictly positive$^{\dag}$, i.e., there exists some positive constant $c>0$ such that the spectrum of $H$ in $L^2(\Bbb H^2,\Bbb C^2)$ is contained in $(c,\infty)$.
Assume further that for some $0<\alpha<2\varrho$, $H$ satisfies$^{\ddag}$
\begin{align}
\|H^{\frac{1}{2}}f\|_{L^2}&\lesssim \|Df\|_{L^2}+\|f\|_{L^2}\label{equi1}\\
\|Df\|_{L^2}&\lesssim \|H^{\frac{1}{2}}f\|_{L^2}+\|f\|_{L^2}\label{equi21}\\
\|e^{-r\alpha}\nabla f\|_{L^2}&\lesssim \|e^{-r\alpha}H^{\frac{1}{2}}f\|_{L^2}+\|e^{-r\alpha}f\|_{L^2}\label{stein23},
\end{align}
provided the right hand sides are finite.
Then if the Kato smoothing effect $\|e^{-\alpha r}e^{\pm it\sqrt{H}}f\|_{L^2_{t,x}}\lesssim \|f\|_{L^2_x}$  holds, we have the following endpoint Strichartz estimates for (\ref{1}): Let $u$ solve (\ref{1}), then for any $p\in(2,6)$
\begin{align*}
&{\left\| D^{\frac{1}{2}}u \right\|_{L_t^2L_x^{p}}} +{\left\| {{e^{-\alpha r}}\nabla u} \right\|_{L_t^2L_x^2}}+
{\left\| {{\partial _t}u} \right\|_{L_t^\infty L_x^2}} + {\left\| {\nabla u} \right\|_{L_t^\infty L_x^2}}\\
&\lesssim {\left\| {\nabla {u_0}} \right\|_{{L^2}}} + {\left\| {{u_1}} \right\|_{{L^2}}} + {\left\| F \right\|_{L_t^1L_x^2}}.
\end{align*}
\end{theorem}
\noindent{\bf Remark 1.1} Generally (\ref{equi1})-(\ref{stein23}) hold if $H$ is a bounded perturbation of $-\Delta$. Even if $H$ has discrete spectrum, we can still expect (\ref{equi1})-(\ref{stein23}) to be right. But if one expects the exact equivalence without the zero order term $\|f\|_{L^2}$, the discrete spectrum of $H$ must be eliminated.

\noindent{\bf Remark 1.2}  If $H$ has discrete spectrum, the Kato smoothing estimates can only hold in the continuous spectrum part of $H$.

The following corollary will show that the Kato smoothing estimates hold for small potentials. Moreover, in the large potential case considered in \cite{Laiazaea}, we can prove the Kato smoothing estimates via choosing a suitable frame on the bundle $Q^*(T\Bbb H^2)$. In fact, the one form $A$ in (\ref{1}) indeed depends on the frame fixed on $Q^*(T\Bbb H^2)$. Then using the geometric setting of the Schr\"odinger operator $H$ and the Coulomb gauge, we can prove no discrete spectrum, no bottom resonance and no embedded eigenvalue exist, which are the enemies in the low frequency and mediate frequency. Moreover, the negative sectional curvature property of the target $\Bbb H^2$ is very important to make the electric potential part be a non-negative operator. Finally, the Kato smoothing effect follows by the decay estimates for the high frequency via choosing a suitable weight and energy arguments. See \cite{Laiazaea} for more details.

\footnotetext{$^\dag$ By Kato's perturbation theorem, (\ref{Vab}) shows $H$ is self-adjoint in $L^2$. See Lemma \ref{povxde45}}
\footnotetext{$^\ddag$ Since $H$ is self-adjoint, $H^{\frac{s}{2}}$ is defined by spectrum theorem. $D:=(-\Delta)^{\frac{1}{2}}$ and can be defined by the Fourier transform on hyperbolic spaces, see Section 2.}

As a corollary we have the endpoint Strichartz estimates for magnetic wave equations in the small potential case.
\begin{corollary}\label{3}
Suppose that the Schr\"odinger operator $H=-\Delta+V+X$ satisfy for some $\varrho>0$
\begin{align}\label{Vab1}
\|e^{r\varrho}V\|_{L^{\infty}}+\|V\|_{L^{2}}+\|e^{r\varrho}|A|\|_{L^{\infty}} +\|A\|_{L^{2}}\le\mu_1.
\end{align}
Assume  $0<\mu_1\ll 1$, $0<\alpha\ll 1$ and $0<\alpha<3\varrho$. If $u$ solves (\ref{1}), then for any $p\in(2,6)$, there holds
\begin{align*}
&{\left\| {{e^{-\alpha r}}\nabla u} \right\|_{{L_t^2L_x^2}}}+\left\| D^{\frac{1}{2}}{u} \right\|_{L_t^2L_x^{p}}+
{\left\| {{\partial _t}u} \right\|_{L_t^\infty L_x^2}} + {\left\| \nabla {u} \right\|_{L_t^\infty L_x^2}}\\
&\lesssim \left\| {\nabla {u_0}} \right\|_{L^2} + \left\| {{u_1}} \right\|_{L^2} + \left\| F \right\|_{L_t^1L_x^2}.
\end{align*}
And as a byproduct, for $s\in[0,\frac{1}{2}]$, $p\in[2,\infty)$, we have
\begin{align}\label{Vab23}
\|D^{2s}f\|_{L^{p}}\thicksim \|H^{{s}}f\|_{L^p}.
\end{align}
\end{corollary}

\noindent{\bf Remark 1.3} (\ref{Vab23}) is useful for studying the well-posedness and scattering of semilinear dispersive equations with magnetic potentials particularly because no chain rule and Leibnitz rule are available for magnetic Schr\"odinger operators $H$.

The key for the proof of Theorem \ref{2} is the weighted Morawetz estimate in Lemma \ref{whishui} and the endpoint weighted Strichartz estimate for free wave equations on $\Bbb H^2$ in Lemma \ref{woshishui}  inspired by \cite{IaKa,DaGa3a}.  The proof of Lemma \ref{woshishui} depends on the bilinear argument of \cite{KaTa}, complex interplation and the frequency decomposition. It is important that Theorem \ref{2} is essentially suitable to potentials of any size.  The key point to involve large potentials in Theorem 1.1 is the three advantages of the hyperbolic background compared with the flat case, i.e., the exotic Strichartz estimates for free wave equations on $\Bbb H^2$, the Kunze-Stein phenomenon, the exponential decay of the spherical functions.  In fact, as noticed by \cite{IaPaSa,AaPa2a}, Strichartz estiamtes of dispersive equations on hyperbolic spaces own more Strichartz pairs than the Euclidean one, see [Corollary 1.3,\cite{KaTa}] and  Lemma \ref{thing} below. For the $\Bbb H^2$ background studied here, an $L^2_t$-type Strichartz estimate is available, which is essential for Lemma \ref{woshishui} and unavailable in the $\Bbb R^2$ case.$^*$

\footnotetext{$^\star$ The Strichartz pair $(p,q)$ in the norm $L^p_tL^q_x$ for wave equations in $\Bbb R^2$ requires $p\ge4$ at leat, see [Corollary 1.3,\cite{KaTa}].}

The key for the proof of Corollary \ref{3} is to prove (\ref{equi1})-(\ref{stein23}).  The Kato smoothing effect in the small potential case is respectively easy. For (\ref{equi1})-(\ref{stein23}), we apply an almost equivalence technique of our previous paper \cite{Laiazaea1a} instead of the heat semigroup techniques usually used in the flat case. In fact, in the Euclidean case (\ref{equi1}) and (\ref{equi21}) are usually proved by Simon's heat semigroup method with Kato's strong Trotter formula, see \cite{Da1a,KaMa,Saiamaoana}. The convenience of our almost equivalence arguments here is that we need neither the commutation property nor the special structure of $H$, which seems to fail for the hyperbolic backgrounds  due to the non-vanishing connection coefficients. In fact, in our argument, it suffices to prove the $L^p$-$L^q$ and weighted $L^2$ resolvent estimates on the half-line $(-\infty,1/4)$. These resolvent estimates can be proved by carefully bounding the resolvent kernel and applying the Kunze-Stein phenomenon.  Moreover, due to the spectrum gap of $-\Delta$, we find this almost equivalence argument directly yields the exact equivalence.

{\bf Notation}
Let $D=(-\Delta)^{\frac{1}{2}}$. The square of $H$ is denoted by $H^{\frac{1}{2}}$ or $\sqrt{H}$. And denote the shifted derivative by $\widetilde{D}=(-\Delta-\frac{1}{4}+{{\kappa}}^2)^{\frac{1}{2}}$ with $\kappa>\frac{1}{2}$. $^\star$

The resolvent of an operator $L$ from one function space to the other is always denoted by $R_{L}(z)=(L-z)^{-1}$ for simplicity, for instance $$R_{H}(z)=(H-z)^{-1},\mbox{  } R_{D}(z)=(D-z)^{-1},\mbox{  }R_{\sqrt{H}}(z)=(\sqrt{H}-z)^{-1}.
$$
In order to coincide with the notions in \cite{BaMa}, we introduce $\mathcal{R}_0(s)$ defined by
\begin{align*}
\mathcal{R}_0(s)=(-\Delta-s(1-s))^{-1}.
\end{align*}
Notice that $s(1-s)$ ranges over all the complex plane if $s$ ranges over the half plane $\{\frac{1}{2}+z:\Re z\ge 0\}$. And if $s$ takes values in the critical line $\{\frac{1}{2}+i\lambda:\lambda\in\Bbb R\}$, then $s(1-s)$ lies in $[1/4,\infty)$ which is the continuous spectrum of $-\Delta$.
Similarly, we introduce the notations: $\mathcal{R}_{D}(z)=(D-z(1-z))^{-1}$, and
$$\mathcal{R}_{H}(z)=(H-z(1-z))^{-1},\mbox{ } \mathcal{R}_{\sqrt{H}}(z)=(\sqrt{H}-z(1-z))^{-1}.
$$

\footnotetext{$^\star$ Since \cite{AaPa,AaPa2a} also considered shifted wave equations, they introduce the operator $\widetilde{D}$ to eliminate the singularity at zero of the corresponding symbols when applying the Fourier transform. The coincidence of $\|D^s\cdot\|_{L^p_x}$ and $\|\widetilde{D}^s\cdot\|_{L^p_x}$ for $p\in(1,\infty)$ makes the use of $\widetilde{D}$ safe for the final estimates and beneficial due to the elimination of singularity at zero. In our case, since the symbol of $(-\Delta)^{-1}$ is $(\lambda^2+1/4)^{-1}$ and has no singularity, it is generally not necessary to introduce $\widetilde{D}$. We keep this notation for reader's convenience of contrasting these papers.}

Let $S(t)$ be any function defined in $\Bbb R$, the notation $S(t)\le t^{-\infty}$ introduced in \cite{AaPa} means for any positive integer $n$, there exists some constant $C(n)>0$ such that $S(t)\le C(n)|t|^{-n}$ when $|t|\to\infty$. Similarly, for any function $S(\cdot)$ defined on integers, $S(j)\le j^{-\infty}$ means for any positive integer $n$, there exists some constant $C(n)>0$ such that $S(j)\le C(n)|j|^{-n}$ when $|j|\to\infty$.

\section{\bf Preliminaries}
Some preliminaries on the geometric notions and the Fourier analysis on the hyperbolic planes are recalled in this section.
Most materials are standard and can be found in Helgason \cite{Haeala}, while some are folk and we will contain some proofs if necessary.

Let $\Bbb D=\{(x_1,x_2)\in\Bbb R^2:|x_1|^2+|x_2|^2<1\}$ be the Poincare model of the hyperbolic plane $\Bbb H^2$ with the metric tension
\begin{align*}
4\frac{dx^2_1+dx^2_2}{(1-|x_1|^2-|x_2|^2)^2}.
\end{align*}
In the polar coordinates $(r,\theta)$, the metric tension of $\Bbb D$ is $d r^2+\sinh^2 rd\theta^2$. The Laplace-Beltrami operator on $\Bbb D$ is
$$
\Delta=\partial^2_r+\coth r\partial_r+\sinh^{-2} r\partial_{\theta}.
$$
The spherical functions $\varphi_{\lambda}$ with $\lambda\in \Bbb C$ on $\Bbb D$  are normalized radial eigenfunctions of $\Delta$:
\begin{align}\label{ojiuyu}
\left\{
  \begin{array}{ll}
    \Delta \varphi_{\lambda}=-(\lambda^2+\frac{1}{4})\varphi_{\lambda}&   \\
    \varphi_{\lambda}(0)=1&
  \end{array}
\right.
\end{align}
For any $\lambda\in\Bbb R$, $r\ge0$, the spherical functions are of exponential decay:
$$
|\varphi_{\lambda}(r)|\le \varphi_{0}(r)\lesssim (1+r)e^{-\frac{r}{2}}.
$$
A horocycle for $\Bbb D$ is a circle in $\Bbb D$ tangential to the boundary $\Bbb B=\partial \Bbb D$. Given $b\in \Bbb B$ and $z\in\Bbb  D$, denote the horocycle through $b$ and $z$ by $\xi(z,b)$. Then we put
\begin{align*}
[z,b]={\rm distance\mbox{  } from\mbox{  } O \mbox{  }to\mbox{  } }\xi(z,b) \mbox{  }(&{\rm with\mbox{  } sign; \mbox{ } to\mbox{  } be\mbox{  } taken \mbox{  }negative}\\
&{\rm if \mbox{  }O \mbox{  }lies \mbox{  }inside}\mbox{  } \xi(z,b)).
\end{align*}
If $f$ is a complex-valued function on $\Bbb D$, the Fourier transform is defined by
\begin{align}
\widetilde{f}(\lambda,b)=\int_{\Bbb D}f(z)e^{(-i\lambda+1)[z,b]}dz,
\end{align}
for all $\lambda\in \Bbb C$, $b\in \Bbb B$, if this integral exists.
The inverse Fourier transform is defined by
\begin{align}
f(z)={\rm const.}\int^{\infty}_{0}\left(\int_{\Bbb B}\widetilde{f}(\lambda,b)e^{(i\lambda+1)[z,b]}db\right)|c(\lambda)|^{-2}d\lambda,
\end{align}
where $c(\lambda)$ is the Harish-Chandra's c-function. The Palancherel formula is
\begin{align}
\|f\|^2_{L^2}={\rm const.}\int^{\infty}_{0}\int_{\Bbb B}|\widetilde{f}(\lambda,b)|^2|c(\lambda)|^{-2} dbd\lambda.
\end{align}
Any function $m:\Bbb R\to \Bbb C$ induces a Fourier multiplier operator $m(-\Delta)$ by the formula $\widetilde{m(-\Delta)f}(\lambda,b)=m(\lambda^2+\frac{1}{4})\widetilde{f}(\lambda,b)$. Thus, the symbol of the operator $m(D)$ is $m(\sqrt{\lambda^2+\frac{1}{4}})$.
Consider the group
$$ {\bf SU}(1,1)=\left\{\left(
                  \begin{array}{cc}
                    a & b \\
                    \bar{b} & \bar{a }\\
                  \end{array}
                \right):|a|^2-|b|^2=1\right\}
$$
with the action on $\Bbb D$ defined by the map
$$
g:z\to\frac{az+b}{\bar{b}z+\bar{a}},\mbox{ } (z\in \Bbb D).
$$
Then we have the identification $\Bbb D={\bf SU}(1,1)/{\bf SO}(2)$.  Let $d\mu(g)$ denote the Haar measure on the group $G={\bf SU}(1,1)$, normalized by
\begin{align}\label{pok9ijn}
\int_{G}f(g\cdot O)d\mu(g)=\int_{\Bbb D}f(z)dz, \forall f\in C_c(\Bbb D),
\end{align}
where $O$ denotes the origin point in $\Bbb D$, $g\cdot O$ denotes the action of $g$ on $O$, and $dz$ denotes the volume form on $\Bbb D$.
For two functions $f_1,f_2$ defined on $\Bbb D$, the convolution denoted by $*$ is defined by
\begin{align}\label{yu78i}
(f_1*f_2)(z)=\int_{G}f_1(g\cdot O)f_2(g^{-1}\cdot z)d\mu(g).
\end{align}
If $k$ is a radial function, then $\widetilde{f*k}=\widetilde{f}\widetilde{k}$.
Since $G$ keeps the Riemannian structure of $\Bbb D$ and Haar measures are invariant under the group action, by (\ref{pok9ijn}) one has the same Young's convolution inequality as the Euclidean space. Furthermore, hyperbolic planes have the so-called Kunze-Stein phenomenon, see Lemma \ref{stein} below. The convolution operation $f_1*f_2$ has an equivalent form if one of $f_1,f_2$ is radial. In fact, suppose that $f_2(x)=f_2(d(x,O))$ is radial, then
\begin{align}
f_1*f_2(z)=\int_{\Bbb D}f_1(x)f_2(d(x,z))dx.\label{yudf78i2}
\end{align}
Indeed, since $f_2$ is radial, $f_2(g^{-1}\cdot z)=f_2(d(g^{-1}\cdot z,O))$. And since $G$ keeps the Riemannian structure of $\Bbb D$, we have
$d(g^{-1}\cdot z,O)=d(z,g\cdot O)$. Then (\ref{yu78i}) reduces to
\begin{align}\label{yu78i2}
(f_1*f_2)(z)=\int_{G}f_1(g\cdot O)f_2(d( z,g\cdot O))d\mu(g).
\end{align}
Considering $f_1(\cdot)f_2(d( z,\cdot))$ as a function in $\Bbb D$, we have (\ref{yudf78i2}) by (\ref{yu78i2}) and (\ref{pok9ijn}).  For any Fourier multiplier operator $m(-\Delta)$, we say the radial function $k(r)$ is the corresponding kernel if $m(-\Delta)f=f*k$. And by (\ref{yudf78i2}), the function $k(d(x,y))$ defined on $\Bbb D\times\Bbb D$ is exactly the Schwartz kernel of $m(-\Delta)$.

Let $\mathcal{H}^{k,p}(\Bbb D;\Bbb R^2)$ be the usual Sobolev space for scalar functions defined on manifolds, see for instance Hebey \cite{Haeabaeaya}. It is known that $C^{\infty}_c(\Bbb D;\Bbb R^2)$ is dense in $\mathcal{H}^{k,p}(\Bbb D;\Bbb R^2)$. We also recall the norm of $\mathcal{H}^{k,p}$:
$$
\|f\|^p_{\mathcal{H}^{k,p}}=\sum^k_{l=1}\|\nabla^l f\|^p_{L^p_x},
$$
where $\nabla^l f$ is the covariant derivative, $p\in[1,\infty)$. The fractional power Sobolev spaces denoted by $\mathcal H^{s,p}$ are defined by $\{f:D^sf\in L^p\}$. And it coincides with $\mathcal{H}^{k,p}$ if $s=k$ is an integer. ${\mathcal H}^{s,2}$ are usually written by $\mathcal{H}^s$
for simplicity.

We now recall the Sobolev inequalities.
\begin{lemma}\label{wusijue}
If $f\in C^{\infty}_c(\Bbb D;\Bbb R)$, then for $1<p<\infty,$ $p\le q\le \infty$, $1<r<2$, $r\le l<\infty$, $\alpha>1$ the following inequalities hold
\begin{align*}
&{\left\| f \right\|_{{L^2}}} \lesssim {\left\| {\nabla f} \right\|_{{L^2}}};  \mbox{  } {\left\| {\nabla f} \right\|_{{L^p}}} \sim{\left\| Df\right\|_{{L^p}}};
 \mbox{  }{\left\| f \right\|_{{L^l}}} \lesssim {\left\| {\nabla f} \right\|_{{L^r}}}\mbox{  }{\rm{when}}\mbox{  }\frac{1}{r} - \frac{1}{2} = \frac{1}{l};\\
&{\left\| f \right\|_{{L^\infty }}} \lesssim {\left\| {{{D}^{{\alpha }}}f} \right\|_{{L^2}}}.
 \end{align*}
\end{lemma}
For the proof, we refer to \cite{IaPaSa,Haeabaeaya,Sataraia}, see also  \cite{LaOaSa}.

The dispersive estimates and Strichartz estimates for free wave equations on $\Bbb H^d$ were considered by \cite{Taataaaraua4a,MaTa1a,MaTa2a,AaPa,AaPa2a}.
Theorem 5.2 and Remark 5.5 of Anker, Pierfelice \cite{AaPa} obtained the Strichartz estimates for linear wave/Klein-Gordon equations. Recall $\widetilde D=(-\Delta-\frac{1}{4}+\kappa^2)^{\frac{1}{2}}$ for some $\kappa>\frac{1}{2}$.
\begin{lemma}[\cite{AaPa}]\label{thing}
Let $(p,q)$ and $(\tilde{p},\tilde{q})$ be two admissible couples, i.e.,
$$\left\{ {({p^{ - 1}},{q^{ - 1}}) \in (0,\frac{1}{2}] \times (0,\frac{1}{2}):\frac{1}{p} > \frac{1}{2}(\frac{1}{2} - \frac{1}{q})} \right\} \cup \left\{ {\left( {0,\frac{1}{2}} \right)} \right\},
$$
and similarly for $(\tilde{p},\tilde{q})$. Meanwhile assume that
$$\sigma  \ge \frac{3}{2}\left( {\frac{1}{2} - \frac{1}{q}} \right),\tilde \sigma  \ge \frac{3}{2}\left( {\frac{1}{2} - \frac{1}{{\tilde q}}} \right),
$$
then the solution $u$ to the linear wave equation
$$
\left\{ \begin{array}{l}
 \partial _t^2u - \Delta u = F \\
 u(0,x) = u_0(x),{\partial _t}\left| {_{t = 0}} \right.u(t,x) = u_1(x) \\
 \end{array} \right.
$$
satisfies the Strichartz estimate
$$ {\left\| {\widetilde D_x^{ - \sigma  + \frac{1}{2}}u} \right\|_{L_t^pL_x^q}} + {\left\| {\widetilde D_x^{ - \sigma  - \frac{1}{2}}{\partial _t}u} \right\|_{L_t^pL_x^q}} \lesssim {\left\| {\widetilde D_x^{\frac{1}{2}}u_0} \right\|_{{L^2}}} + {\left\| {\widetilde D_x^{ - \frac{1}{2}}u_1} \right\|_{{L^2}}} + {\left\| {\widetilde D_x^{\tilde \sigma  - \frac{1}{2}}F} \right\|_{L_t^{\tilde p'}L_x^{\tilde q'}}}.
$$
\end{lemma}

\noindent{\bf Remark}
For all $\sigma\in \Bbb R$, $p\in(1,\infty)$, $\|\widetilde D^{\sigma}f\|_{p}$ is equivalent to $\|D^{\sigma}f\|_{p}$, see [Page 5618,\cite{AaPa2a}].

\section{Proof of Theorem 1.1}
\begin{lemma}\label{povxde45}
Let $B,A$ satisfy (\ref{Vab}), then $H$ is a self-adjoint operator in $L^2(\Bbb D;\Bbb C^2)$ with domain $D(H)=\mathcal{H}^2$.
\end{lemma}
\begin{proof}
Since we will work with $\Bbb C^2$-valued functions, the operators $d,d^*$ are assumed to act on $\Omega^{p}(\Bbb D)\otimes\Bbb C^2$ and $\Omega^{p}(\Bbb D)\otimes gl(2,\Bbb C)$.
First, $-\Delta$ is self-adjoint in $L^2$ with domain $\mathcal{H}^2$ (see for instance \cite{Sataraia}). Second, $H$ is symmetric in $L^2(\Bbb D;\Bbb C^2)$ with domain $\mathcal{H}^2$ by integration by parts. In fact, denote the inner product in $\Bbb C^2$ by $\langle \cdot,\cdot\rangle$.
Given, $f,g\in C^{\infty}_c(\Bbb D:\Bbb C^2)$. Since $B$ is symmetric and real, it is obvious that
\begin{align}
\int_{\Bbb D}\langle Bf,g\rangle dz=\int_{\Bbb D}\langle f,Bg\rangle dz.
\end{align}
Since in the dimension two $d^*=-*d*$, one can check $d^*(Af)=(d^*A)f-(A,df)$ with $(\cdot,\cdot)$ denoting the metric on $\Omega^{1}(\Bbb D)$. Then one has $2(A,df)-(d^*A)f=-d^*(Af)+(A,df)$. Since $A$ is real and anti-symmetric, $\langle(A,df),{g}\rangle=- (df,Ag)_1$, where $(\cdot,\cdot)_1$ denotes the metric in $\Omega^{1}(\Bbb D)\otimes\Bbb C^2$. Thus by the above three identities and integration by parts,
\begin{align*}
&\int_{\Bbb D}2\langle (A,df),g\rangle-\langle (d^*A)f,g\rangle  dz
=\int_{\Bbb D}\langle (A,df),g\rangle dz-\int_{\Bbb D}\langle d^*(Af),g\rangle dz\\
&=-\int_{\Bbb D} (df,Ag)_1dz-\int_{\Bbb D}(Af,dg)_1dz
=-\int_{\Bbb D} \langle f,d^*(Ag)\rangle dz+\int_{\Bbb D}\langle f, (A,dg)\rangle dz\\
&=-\int_{\Bbb D}\langle f,(d^*A)g)\rangle dz+2\int_{\Bbb D}\langle f, (A,dg)\rangle dz.
\end{align*}
Third, $H$ is a compact perturbation of $-\Delta$: For any $\epsilon>0$ and $K>0$ to be determined later, the compactness of Sobolev embedding in bounded geodesic balls implies there exists some $C_1(\epsilon,K)$ such that
\begin{align}\label{p9uihbgy}
\|(V+X)f\|_{L^2(B_K)}\le \epsilon\|\Delta f\|_{L^2(B_K)}+C_1(\epsilon,K)\|f\|_{L^2(B_K)},
\end{align}
where $B_K$ denotes the geodesic ball with center $O$ of radius $K$.
By taking $K\gg1$, the exponential decay of $V$ and $|A|$ (see (\ref{Vab})) yields
\begin{align}\label{puihbgy}
\|(V+X)f\|_{L^2({\Bbb D}\backslash(B_K))}\lesssim e^{-\varrho K}\|f\|_{L^2}+e^{-\varrho K}\|\nabla f\|_{L^2}.
\end{align}
Thus the Poincare inequality $\|f\|_{L^2}+\|\nabla f\|_{L^2}\lesssim \|\Delta f\|_{L^2}$ for hyperbolic spaces with (\ref{p9uihbgy}), (\ref{puihbgy}) implies
\begin{align}
\|(V+X)f\|_{L^2(\Bbb D)}\lesssim (e^{-\varrho K}+\epsilon)\|\Delta f\|_{L^2}+C_1(\epsilon,K)\|f\|_{L^2}.
\end{align}
Let $K$ be sufficiently large, by Kato's perturbation theorem, $H$ is self-adjoint.
\end{proof}

Since $H$ is assumed to be positive, one can define its fractional power $H^{s}$  for any $s\in \Bbb R$ via the spectrum theorem.

For reader's convenience, we recall the following lemma of \cite{AaPa2a} whose proof is based on the Kunze-Stein phenomenon.
\begin{lemma}[Lemma 5.1,\cite{AaPa2a}]\label{stein}$^*$
There exists a constant $C>0$ so that for any radial function $h$ on $\Bbb D$, any $2\le m,k<\infty$ and $g\in L^{k'}(\Bbb D)$,
\begin{align*}
\| {g * h} \|_{{L^m}} \lesssim \left\| g \right\|_{L^{k'}}\left\{ \int_0^\infty  \sinh r(\varphi _0{(r)})^{P}|h(r)|^Q dr \right\}^{1/Q},
\end{align*}
where $P= \frac{2\min \{m,k\} }{{m + k}}$, $Q= \frac{{mk}}{{k + m}}$, and $\varphi _0$ is the spherical function defined in (\ref{ojiuyu}).
\end{lemma}

\footnotetext{$^*$ The JDE version of [Lemma 5.1,\cite{AaPa2a}] has some misprints. And we take this Lemma from its arxiv version which is correct. }

The following weighted Strichartz estimate will be important to prove the endpoint Strichartz estimates. Its proof is an application of the smoothing effects. Recall $\rho(x)=e^{-d(x,0)}$.
\begin{lemma}\label{whishui}
Let $H$ satisfy assumptions in Theorem \ref{1}. Assume that $u$ solves
\begin{align}\label{google}
\left\{ \begin{array}{l}
\partial _t^2u + Hu =F \\
u(0,x) = {u_0},{\partial _t}u(0,x) = {u_1} \\
\end{array} \right.
\end{align}
Then we have
\begin{align}
{\left\| {{\rho ^{\alpha }}\nabla u} \right\|_{L_t^2L_x^2}} \lesssim  {\left\| F(t)\right\|_{L_t^1L_x^2}} + {\left\| {\nabla {u_0}} \right\|_{L_x^2}} + {\left\| {{u_1}} \right\|_{L_x^2}}.\label{gutu4}
\end{align}
\end{lemma}
\begin{proof}
The proof is an easy application of the Kato's smoothing effect of $e^{\pm it\sqrt{H}}$. In fact, by Duhamel principle,
$$
u(t) = \cos \left( {t\sqrt { H} } \right){u_0} + \frac{\sin \left( {t\sqrt { H} } \right)}{\sqrt{H}}{u_1} + \int_0^t {\frac{{\sin \left( {(t - s)\sqrt { H} } \right)}}{{\sqrt { H} }}} F(s)ds.
$$
By the Christ-Kiselev lemma, for the inhomogeneous term it suffices to prove
\begin{align*}
\left\|\int_{\Bbb R} \rho^{\alpha}\nabla H^{-\frac{1}{2}}\sin \left( {(t - s)\sqrt { H} } \right) F(s)ds\right\|_{L^2_tL^2_x}\lesssim \|F\|_{L^1_tL^2_x}.
\end{align*}
By (\ref{stein23}), the Kato's smoothing effect and Minkowski,
\begin{align}
&\left\|\int_{\Bbb R} \rho^{\alpha}\nabla H^{-\frac{1}{2}}\sin \left( {(t - s)\sqrt { H} } \right) F(s)ds\right\|_{L^2_tL^2_x}\nonumber\\
&\lesssim\int_{\Bbb R}\left\|\rho^{\alpha}\sin \left( {(t - s)\sqrt { H} } \right) F(s)\right\|_{L^2_tL^2_x}ds\nonumber\\
&+\int_{\Bbb R}\left\|\rho^{\alpha}H^{-\frac{1}{2}}\sin \left( {(t - s)\sqrt { H} } \right) F(s)\right\|_{L^2_tL^2_x}ds\nonumber\\
&\lesssim\int_{\Bbb R}\|F(s)\|_{L^2_x}ds+\int_{\Bbb R}\|H^{-\frac{1}{2}}F(s)\|_{L^2_x}ds.\label{po9nbgh}
\end{align}
Meanwhile, the strict positiveness of the self-adjoint operator $H$ and the spectrum theorem  imply
\begin{align}\label{poiunvfg}
\|H^{-\frac{1}{2}}F\|^2_{L^2_x}=\langle H^{-1}F,F\rangle\le c\|F\|^2_{L^2_x}.
\end{align}
Hence the estimates for the inhomogeneous term follow by (\ref{poiunvfg}) and (\ref{po9nbgh}).
Similarly, the two homogeneous terms are bounded by $\|\nabla u_0\|_{L^2}+\|u_1\|_{L^2}$ by applying (\ref{stein23}), Kato's smoothing effect for $e^{\pm it\sqrt{H}}$ and the standard Poincare inequality $\|f\|_{L^2}\lesssim \|\nabla f\|_{L^2}$.
\end{proof}

The proof of non-endpoint Strichartz estimates is quite standard.
The non-endpoint homogeneous Strichartz estimates are given below.
\begin{lemma}\label{Smoothing}.
Let $(p,q)$ be an admissible pair with $p>2$, then
\begin{align}
{\left\| D^{\frac{1}{2}}e^{ \pm it\sqrt { H}}f \right\|_{L_t^pL_x^q}} &\lesssim \left\| D f\right\|_{L_x^2}\label{smoothing}. \\
\left\|D^{\frac{1}{2}}e^{ \pm it\sqrt { H}}f \right\|_{L_t^pL_x^q} &\lesssim \left\| H^{\frac{1}{2}}f\right\|_{L_x^2}.\label{wavemap2}
\end{align}
\end{lemma}

\begin{proof}
We follow the framework of \cite{DG2}.
Recall that $W=V+X$ is the potential part of $H$.
Denote ${e^{it\sqrt { H} }}f=u$, then
\begin{align}\label{p9ij}
e^{it\sqrt { H}}f = \cos \left( {tD} \right)f + i\frac{\sin D}{D}\sqrt { H} f - \int_0^t \frac{\sin \left( {(t - s)D} \right)}{D} Wuds.
\end{align}
Lemma \ref{thing}, (\ref{equi1}) and Lemma \ref{wusijue} show
$${\left\| {\frac{\sin \left( {tD} \right)}{D^{\frac{1}{2}}}\sqrt {H} f} \right\|_{L_t^pL_x^q}} \lesssim {\left\| {\sqrt { H} f} \right\|_{L_x^2}} \lesssim {\left\| {D}f \right\|_{{L^2_x}}}.
$$
Thus the homogeneous estimate is done.
The rest is to handle the inhomogeneous term. As a preparation, we first prove
\begin{align}\label{9f}
{\left\| {\int_0^t {\frac{{\sin \left( {(t - s)D} \right)}}{{{{D}^{\frac{1}{2}}}}}} Wuds} \right\|_{L_t^pL_x^q}} \lesssim  {\left\| {{\rho ^\alpha }H^{\frac{1}{2}} u} \right\|_{L_t^2L_x^2}}+{\left\| {{\rho ^\alpha } u}\right\|_{L_t^2L_x^2}}.
\end{align}
Since $p>2$, by the Christ-Kiselev lemma, to verify (\ref{9f}) it suffices to prove
\begin{align}
\left\| \int_{\Bbb R} \frac{e^{\pm i(t - s)D}}{D^{\frac{1}{2}}} Wuds \right\|_{L_t^pL_x^q} \lesssim \left\| \rho ^\alpha H^{\frac{1}{2}}u \right\|_{L_t^2L_x^2}+\left\| {{\rho ^\alpha } u}\right\|_{L_t^2L_x^2}.
\end{align}
Recall the Kato smoothing effect for $e^{iD t}$: for any $g\in L^2$ there holds
$$
{\left\| {\rho ^\alpha }{e^{\pm it D}g}\right\|_{L_t^2L_x^2}} \lesssim {\left\| g \right\|_{L_x^2}}.
$$
The dual version is
\begin{align}\label{nji}
 {\left\| {\int_{\Bbb R} {{e^{ \mp i D}}} F(s)ds} \right\|_{L_x^2}} \lesssim {\left\| {{\rho ^{ - \alpha }}F} \right\|_{L_t^2L_x^2}}.
\end{align}
(\ref{nji}) and Lemma \ref{thing} give
\begin{align*}
{\left\| {\int_{\Bbb R} {{D}^{ - \frac{1}{2}}}{e^{\pm i(t - s)D} }Wuds} \right\|_{L_t^pL_x^q}} &\lesssim {\left\| {\int_{\Bbb R} {e^{ - isD}Wuds}} \right\|_{L_x^2}} \\
&\lesssim {\left\| {{\rho ^{-\alpha} }\left( {{V} + X} \right)u} \right\|_{L_t^2L_x^2}}.
\end{align*}
Thus by (\ref{Vab}) and (\ref{stein23}), one deduces
\begin{align*}
 &{\left\| {\int_{\Bbb R} {{{D}^{ - \frac{1}{2}}}{e^{\pm i (t - s)D}}}Wuds} \right\|_{L_t^pL_x^q}} \\
 &\lesssim \left( {{{\left\| {{V}{\rho ^{ - 2\alpha }}} \right\|}_{{L^\infty_x }}} + {{\left\| {|A|{\rho ^{ - 2\alpha }}} \right\|}_{{L^\infty_x }}}} \right)\left( {{{\left\| {{\rho ^\alpha }u} \right\|}_{L_t^2L_x^2}} + {{\left\| {{\rho ^\alpha }\nabla u} \right\|}_{L_t^2L_x^2}}} \right) \\
 &\lesssim {\left\| {{\rho ^\alpha }u} \right\|_{L_t^2L_x^2}} + {\left\| {{\rho ^\alpha }{{H}^{\frac{1}{2}}}u} \right\|_{L_t^2L_x^2}}.
\end{align*}
Hence (\ref{9f}) has been proved.
Since ${{{H}^{\frac{1}{2}}}}$ commute with $e^{\pm it\sqrt {H}}$, (\ref{9f}), the Kato's smoothing effect for  $e^{\pm it\sqrt{H}}$ and (\ref{equi1})-(\ref{equi21})  yield
\begin{align*}
&{\left\| {\int^t_{0} {{{ D}^{ - \frac{1}{2}}}{e^{\pm i(t - s)D}}} Wuds} \right\|_{L_t^pL_x^q}}\\
&\lesssim\left\| \rho ^\alpha H^{\frac{1}{2}}u \right\|_{L_t^2L_x^2}+\left\| {{\rho ^\alpha } u}\right\|_{L_t^2L_x^2}\\
&\lesssim {\left\| f \right\|_{L_x^2}} + {\left\| {{{H}^{\frac{1}{2}}}f} \right\|_{L_x^2}} \lesssim {\left\| {\nabla f} \right\|_{L_x^2}}.
\end{align*}
Thus we have obtained (\ref{smoothing}). (\ref{wavemap2}) follows by (\ref{smoothing}) and (\ref{poiunvfg}).
\end{proof}

\begin{proposition}\label{jiu9}
Let $H$ satisfy the assumptions in Theorem \ref{1}. Then we have the non-endpoint Strichartz estimates for the magnetic wave equation: If $u$ solves the equation
\begin{align}
\left\{ \begin{array}{l}
\partial _t^2u +Hu = F \\
u(0,x) = {u_0},{\partial _t}u(0,x) = {u_1} \\
\end{array} \right.
\end{align}
then it holds for any admissible pair $(p,q)$ with $p>2$, $q\in(2,6]$
\begin{align*}
&{\left\| D^{\frac{1}{2}}u \right\|_{L_t^pL_x^q}} + {\left\| D^{-\frac{1}{2}}\partial_tu \right\|_{L_t^pL_x^q}}+{\left\| {{\partial _t}u} \right\|_{L_t^\infty L_x^2}} + {\left\| {\nabla u} \right\|_{L_t^\infty L_x^2}} \\
&\lesssim {\left\| {\nabla {u_0}} \right\|_{{L^2}}} + {\left\| {{u_1}} \right\|_{{L^2}}}+{\left\| F \right\|_{L_t^1L_x^2}}.
\end{align*}
\end{proposition}
\begin{proof}
By Duhamel principle,
$$
u(t) = \cos \left( {t\sqrt { H} } \right){u_0} + \frac{1}{\sqrt { H}}\sin \left( {t\sqrt { H} } \right){u_1} + \int_0^t {\frac{{\sin \left( {(t - s)\sqrt { H} } \right)}}{{\sqrt { H} }}} F(s)ds.
$$
The homogenous estimates follow directly by Lemma \ref{Smoothing}, (\ref{equi1})-(\ref{equi21}), and the inequality $\|f\|_{L^2}\lesssim \|(-\Delta)^{s}f\|_{L^2}$ for any $s\in(0,1)$.
It remains to deal with the inhomogeneous term. By the Christ-Kiselev lemma, it suffices to prove
\begin{align*}
{\left\| {\int_{\Bbb R} {{D^{\frac{1}{2}}}{{H}^{ -\frac{1}{2}}}\sin \left( {(t - s)\sqrt { H} } \right)F(s)ds} } \right\|_{L_t^pL_x^q}} \lesssim {\left\| F \right\|_{L_t^1L_x^2}}.
\end{align*}
This is an immediate corollary of (\ref{equi1})-(\ref{equi21}) and (\ref{wavemap2}). In fact, we have by (\ref{wavemap2}) and Minkowski inequality
\begin{align*}
 &{\left\| {\int_{\Bbb R} {{D^{\frac{1}{2}}}{{H}^{ - \frac{1}{2}}}\sin \left( {(t - s)\sqrt { H} } \right)F(s)ds} } \right\|_{L_t^pL_x^q}} \\ &\lesssim\int_{\Bbb R} {{{\left\| {{D^{\frac{1}{2}}}{H^{ - \frac{1}{2}}}{e^{ \pm i(t - s){\sqrt{H}} }}}F(s)\right\|}_{L_t^pL_x^q}}} ds\\
 &\lesssim\int_{\Bbb R} \left\|H^{\frac{1}{2}}{H^{ - \frac{1}{2}}}{e^{ \pm is\sqrt{H}}}F(s) \right\|_{L_x^2} ds
 \lesssim\left\| F(s)\right\|_{L^1_sL_x^2}.
\end{align*}
The estimate of $\partial_tu$ is similar.
\end{proof}

Recall $\widetilde{D}=(-\Delta-\frac{1}{4}+{{\kappa}}^2)^{\frac{1}{2}}$ in Section 1. Let $\chi_{\infty}(r)$ be a cutoff function which equals one when $r\ge3/2$ and vanishes near zero. In the vertical strip $0\le \Re\sigma\le 3/2$, we define an analytic family of operators
\begin{align}\label{poui878}
\widetilde{W}^{\sigma,\infty}_t=\frac{e^{\sigma^2}}{\Gamma(\frac{3}{2}-\sigma)}\chi_{\infty}(D)\widetilde{D}^{-\sigma}e^{itD}
\end{align}
Denote its kernel by $\widetilde{w}^{\sigma,\infty}_t(r)$. It is easily seen that $\widetilde{W}^{\sigma,\infty}_t$ is the high frequency truncation of $
\widetilde{D}^{-\sigma}e^{itD}$. The Gamma function added to (\ref{poui878}) allows us to handle the boundary $\Re\sigma=\frac{3}{2}$ (see \cite{AaPa}). For $\sigma\in\Bbb R$, define the low frequency truncation of $\widetilde{D}^{-\sigma}e^{itD}$ to be
\begin{align}\label{poui87t}
{W}^{\sigma,0}_t=\widetilde{D}^{-\sigma}e^{itD}(I-\chi_{\infty}(D)).
\end{align}
Denote its kernel by $w^{\sigma,0}_t(r)$.
We collect some results from [Section 3, \cite{AaPa}] for reader's convenience.
\begin{lemma}[\cite{AaPa}]\label{wjiu87}
The kernel $w^{\sigma,0}_t(r)$ satisfies the point-wise estimates for $\sigma\in\Bbb R$, $|t|\ge2$
\begin{align}\label{hterxcfd3}
\left| {w_t^{\sigma,0}(r)} \right| \lesssim \left\{ \begin{array}{l}
 {\left| t \right|^{ - 3/2}}(1 + r){\varphi _0}(r),\mbox{  }0 \le r \le \frac{1}{2}\left| t \right| \\
 {\varphi _0}(r),\mbox{  }r \ge \frac{1}{2}\left| t \right| \\
 \end{array} \right.
\end{align}
And for $\sigma\in\Bbb C$ with $\Re \sigma=\frac{3}{2}$, the kernel $w^{\sigma,\infty}_t(r)$ satisfies
\begin{align}\label{hrxcfd3}
\left| {\widetilde{w}_t^{\sigma ,\infty }(r)} \right| \lesssim \left\{ \begin{array}{l}
 {\left| t \right|^{ - \infty }}{\varphi _0}(r),\mbox{  }0 \le r \le \frac{1}{2}\left| t \right| \\
 {e^{ - \frac{1}{2} r}}t,\mbox{  }r \ge \frac{1}{2}\left| t \right| \\
 \end{array} \right.
\end{align}
\end{lemma}
\begin{remark}
{\rm (\ref{hterxcfd3}) can be found in Theorem 3.1 of \cite{AaPa}. (\ref{hrxcfd3}) is contained in the proof of Theorem 3.2 in \cite{AaPa}}.
\end{remark}

Lemma \ref{wjiu87} has several corollaries.
\begin{corollary}\label{hasonggu}
Assume that $p\in(2,6)$, $q\in(2,6)$, then for $t\ge2$
\begin{align}\label{masu7}
{\left\| {f * w_t^{\sigma ,0}(r)} \right\|_{{L^p}}} \lesssim {t^{ - 3/2}}{\left\| f \right\|_{{L^{q'}}}}.
\end{align}
Moreover, for $\Re\sigma> 3(\frac{1}{2}-\frac{1}{p})$, $\Re\sigma>3(\frac{1}{2}-\frac{1}{q})$, $t\ge2$, one has
\begin{align}\label{6.13u}
{\left\| {f * {\widetilde{w}^{\sigma,\infty }_t}} \right\|_{{L^{p}}}} \lesssim {t^{ - \infty }}{\left\| f \right\|_{{L^{q'}}}}.
\end{align}
\end{corollary}
\begin{proof}
First, we prove the $w^{\sigma,0}_t(r)$ part in (\ref{masu7}).  When $r \le \frac{1}{2}|t|$ the desired estimate in (\ref{masu7}) follows by
applying Lemma \ref{stein} and (\ref{hterxcfd3}). When $r \ge \frac{1}{2}|t|$, since ${\varphi _0}(r)\lesssim (1+r)e^{-\frac{r}{2}}$, choosing arbitrary $0<\epsilon\ll1$, we obtain
$$ {\varphi _0}(r)\lesssim t^{-\frac{3}{2}}e^{-\frac{1-\epsilon}{2}r}.
$$
Then the desired estimate in (\ref{masu7}) follows by applying Lemma \ref{stein} as well. \\
Second, we deal with the $\widetilde{w}^{\sigma,\infty}_t(r)$ part in (\ref{6.13u}). This is achieved by interpolation. In fact, for $\Re\sigma=0$
\begin{align}\label{masu9}
{\left\| f * \widetilde{w}^{\sigma,\infty }_t \right\|_{{L^2}}} \le C {\left\| f \right\|_{{L^2}}}.
\end{align}
For $\Re\sigma=3/2$, as in the first step, (\ref{hrxcfd3}) with Lemma \ref{stein} yields for any $2<m<\infty$, $k\in[2,\infty)$
\begin{align}\label{masu10}
{\left\| f * \widetilde{w}^{\sigma,\infty  }_t \right\|_{{L^k}}} \le C{t^{ - \infty }}{\left\| f \right\|_{{L^{m'}}}}.
\end{align}
Interpolating (\ref{masu9}) with (\ref{masu10}), for $\frac{1}{p}=\frac{\theta}{2}+\frac{1-\theta}{k}$ $\frac{1}{q'}=\frac{\theta}{2}+\frac{1-\theta}{m'}$, and $\Re\sigma=\frac{3}{2}(1-\theta)$, we have
\begin{align*}
{\left\| {f * \widetilde{w}^{\sigma,\infty  }_t} \right\|_{{L^p}}} \le C{t^{ - \infty }}{\left\| f \right\|_{{L^{q'}}}}.
\end{align*}
In this case, by checking the relations $k=\frac{2}{3}\sigma/(\frac{1}{p}-\frac{1}{2}+\frac{1}{3}\sigma)$, $m'=\frac{2}{3}\sigma/(\frac{1}{q'}-\frac{1}{2}+\frac{1}{3}\sigma)$, and $2<m<\infty$, $k\in[2,\infty)$, we conclude
that for $\Re\sigma> 3(\frac{1}{2}-\frac{1}{p})$, $\Re\sigma> 3(\frac{1}{2}-\frac{1}{q})$, $p\in(2,6)$, $q\in(2,6)$, there holds
\begin{align*}
\|f*\widetilde{w}^{\sigma,\infty  }_t\|_{L^p}\le  C(p,q)t^{-\infty}\|f\|_{L^{q'}}.
\end{align*}
\end{proof}

The analogy of the following lemma for Schr\"odinger operators in the Euclidean spaces is obtained by  \cite{IaKa} which plays an important role in the endpoint argument of \cite{DaGa3a}.
\begin{lemma} \label{woshishui}
Let $u$ solve the linear wave equation
\begin{align}\label{google2}
\left\{ \begin{array}{l}
\partial _t^2u - {\Delta}u =F \\
u(0,x) = 0,{\partial _t}u(0,x) = 0 \\
\end{array} \right.
\end{align}
Let $\alpha>0$, then for $q\in(2,6)$
\begin{align}\label{ky6tfcser}
{\left\| {{{D}^{\frac{1}{2}}}u} \right\|_{L_t^2L_x^q}} \le {\left\| {{\rho ^{ - \alpha }}F} \right\|_{L_t^2L_x^2}}.
\end{align}
\end{lemma}
\begin{proof}
{\bf Step 1. A Non-endpoint Result.}
 Although the Christ-Kiselev Lemma is not available here, we can firstly prove a non-endpoint result, i.e.,
\begin{align}\label{masu1}
{\left\| {{D^{\frac{1}{2}}}u} \right\|_{L_t^pL_x^q}} \le {\left\| {{\rho ^{ - \alpha }}F} \right\|_{L_t^2L_x^2}},
\end{align}
where $(p,q)$ is an admissible pair and $p>2$. The proof of (\ref{masu1}) follows from the Christ-Kiselev lemma, Lemma \ref{thing} and the Kato's smoothing effect.\\
{\bf Step 2. Bilinear Argument for Endpoint.}
 Next, we prove the endpoint case. The proof is based on the bilinear argument of \cite{KaTa}.\\
{\bf Step 2.1. Reduction of Time Support.} By Duhamel principle,
\begin{align}\label{masu1u}
D^{\frac{1}{2}}u(t) =\int^t_0 \frac{\sin (D(t-s))}{D^{\frac{1}{2}}}F(s)ds.
\end{align}
Meanwhile, the endpoint Strichartz estimates for the free wave equation in Lemma \ref{thing} and the dual Kato smoothing effect for $e^{\pm it D}$ show
\begin{align*}
\left\|\int^0_{-\infty} \frac{\sin (D(t-s))}{D^{\frac{1}{2}}}F(s)ds\right\|_{L_t^2L_x^q}\lesssim \left\|\int^{0}_{-\infty} e^{\pm isD} F(s)ds \right\|_{L^2_x}\lesssim \|\rho ^{ - \alpha }F\|_{L^2_tL^2_x}.
\end{align*}
Hence, (\ref{ky6tfcser}) reduces to prove
\begin{align}\label{masu1u678}
\left\|\int^t_{-\infty} \frac{\sin (D(t-s))}{D^{\frac{1}{2}}}F(s)ds\right\|_{L_t^2L_x^q}\lesssim \|\rho ^{ - \alpha }F\|_{L^2_tL^2_x}.
\end{align}
Consider the bilinear form of (\ref{masu1u678}):
\begin{align}\label{masu2}
\left| {\int_{\Bbb R} {\int_{{\Bbb D}} {\int_{-\infty}^t {{e^{i(t - s)D}}D^{-\frac{1}{2}}F(s)G(t)dsdxdt} } } } \right| \lesssim {\left\| {{\rho ^{ - \alpha }}F(t)} \right\|_{L_t^2L_x^2}}{\left\| {G(t)} \right\|_{L_t^2L_x^{q'}}}.
\end{align}
Split the time integrand domain into the following dyadic subintervals:
\begin{align}
&\int_{\Bbb R}{\int_{{\Bbb D}} {\int_{-\infty}^t {{e^{i(t - s)D}}D^{-\frac{1}{2}}F(s)G(t)dsdxdt} }}\nonumber \\
&= \sum\limits_{j \in Z} {\int_{\Bbb R} {\int_{{\Bbb D}} {\int_{t - {2^j} \le s \le t - {2^{j- 1}}} {{e^{i(t - s)D}}D^{-\frac{1}{2}}F(s)G(t)} dsdxdt} }}.\label{6.22p}
\end{align}
For any fixed $j\in\Bbb Z$, divide $F(s)$, $G(t)$ further into $F(s)=\sum_{k\in\Bbb Z}F^j_k(s)$, $G(t)=\sum_{n\in\Bbb Z}G^j_n(t)$ with
\begin{align*}
F^j_k(s)=F(s)1_{s\in[k2^j,(k+1)2^j)}, \mbox{ }G^j_n(t)=G(t)1_{t\in[n2^j,(n+1)2^j)}.
\end{align*}
When $s\in[k2^j,(k+1)2^j)$, $k\in \Bbb Z$, the relation $t - {2^j} \le s \le t - {2^{j- 1}}$ shows $|n-k|\le 2$. If we have proved there exists a constant $\beta(q)>0$ such that for all $k,j\in \Bbb Z$
\begin{align}
&{\int_{\Bbb R} {\int_{\Bbb D} {\int_{t - {2^j} \le s \le t - {2^{j- 1}}, |n-k|\le2} {{e^{i(t - s)D}}D^{-\frac{1}{2}}F^j_k(s)G^j_n(t)dsdxdt} }}}\nonumber \\
&\lesssim |j|^{-\beta}\| \rho ^{ - \alpha }F^j_k \|_{L_t^2L_x^2}\| G^j_n\|_{L_t^2L_x^{q'}}.\label{y675}
\end{align}
Then the Cauchy-Schwartz inequality and the restriction $|n-k|\le 2$ give
\begin{align*}
|(\ref{6.22p})|&\lesssim \sum_{j\in\Bbb Z}|j|^{-\beta}\sum_{k,n\in\Bbb Z, |n-k|\le 2}\|F^j_k\|_{L_t^2L_x^2}\left\| G^j_n\right\|_{L_t^2L_x^{q'}}\\
&\lesssim \sum_{j\in\Bbb Z}|j|^{-\beta}(\sum_{k\in \Bbb Z}{\|F^j_k\|^2_{L_t^2L_x^2}})^{\frac{1}{2}}(\sum_{n\in \Bbb Z}\|G^j_n\|^2_{L_t^2L_x^{q'}})^{\frac{1}{2}}\\
&\lesssim \|F\|_{L_t^2L_x^2}\|G\|_{L_t^2L_x^{q'}}\sum_{j\in\Bbb Z}|j|^{-\beta}.
\end{align*}
Thus it suffices to prove (\ref{y675}).
Therefore, without loss of generality, we assume $F$ and $G$ are supported on a time interval of size $2^j$ on $\{(t,s):{t - {2^j} \le s \le t - {2^{j- 1}}}\}$.\\
{\bf Step 2.2. Sum of Negative $j$.}
 For $j\le0$, $q\in (2,6]$, choose $m$ to be slightly larger than 2, then  H\"older and (\ref{masu1}) give for $\frac{1}{m}>\frac{1}{2}(\frac{1}{2}-\frac{1}{q})$ (then $(m,q)$ is an admissible pair)
\begin{align}
 &\left| {\int_{\Bbb R} {\int_{{\Bbb D}} {\int_{t - {2^j} \le s \le t - {2^{j- 1}}} {{e^{i(t - s)D}}D^{-\frac{1}{2}}F(s)G(t)} dsdxdt} } } \right| \nonumber\\
 &\lesssim {\left\| {\int_{t - {2^j} \le s \le t - {2^{j - 1}}} {{e^{i(t - s)D}}D^{-\frac{1}{2}}F(s)} ds} \right\|_{L_t^{m}L_x^q}}{\left\| {G(t)} \right\|_{L_t^{m'}L_x^{q'}}} \nonumber\\
 &\lesssim {\left\| {{\rho ^{ - \alpha }}F(t)} \right\|_{L_t^2L_x^2}}{\left\| {G(t)} \right\|_{L_t^{m'}L_x^{q'}}} \nonumber\\
 &\lesssim {\left\| {{\rho ^{ - \alpha }}F(t)} \right\|_{L_t^2L_x^2}}{\left\| {G(t)} \right\|_{L_t^2L_x^{q'}}}{2^{j(\frac{1}{m'} - \frac{1}{2})}},\label{yu76}
\end{align}
where we used the time support of $G(t)$ is of size $2^j$ in the last line.  Therefore, the negative $j$ part of (\ref{6.22p}) is summable. \\

{\bf Step 2.3. Sum of Positive $j$.}   For the positive $j$, let us consider a multiple parameter analytic family of operators defined by
\begin{align}\label{1suiyue}
T_j^{\sigma ,{\sigma _1}}(F,G) = \int_{{\Bbb R}} {\int_{\Bbb D}} {\int_{t - {2^j} \le s \le t - {2^{j - 1}} }}{{e^{i(t - s)D}} } { {D}^{ - \sigma }}{\widetilde{D}^{ - {\sigma _1}}}F(s)G(t)dsdxdt.
\end{align}
By Remark 2.1, every estimate for $T_j^{\sigma ,{\sigma _1}}(F,G)$ will yield a corresponding estimate for $T_j^{\sigma+\sigma_1 ,0}(F,G)$ if both $\sigma$ and $\sigma_1$ are real.
Furthermore, we divide $T_j^{\sigma ,{\sigma _1}}$ into $T^{\sigma,\sigma_1}_{j,\infty}+T^{\sigma,\sigma_1}_{j,0}$, where  $T^{\sigma,\sigma_1}_{j,\infty}=\chi_{\infty}(D)T_j^{\sigma ,{\sigma _1}}$ denotes the high frequency part (see Lemma \ref{wjiu87} for $\chi_{\infty}$), and $T^{\sigma,\sigma_1}_{j,0}=T^{\sigma,\sigma_1}_{j}(I-\chi_{\infty})(D)$ denotes the low frequency part.\\

{\bf Step 2.3.1. High Frequency for Positive  $j$.}
We claim for any $\Re\sigma> \frac{3}{2}(\frac{1}{2}-\frac{1}{q}),$ $\Re\sigma_1> \frac{3}{2}(\frac{1}{2}-\frac{1}{p})$, $(p,q)\in(2,6)\times(2,6)$
\begin{align}\label{i987}
\left| {T_{j,\infty }^{\sigma ,{\sigma _1}}(F,G)} \right| \lesssim {j^{ - \infty }}{\left\| F \right\|_{L_t^2L_x^{p'}}}{\left\| G \right\|_{L_t^2L_x^{q'}}}.
\end{align}
(\ref{i987}) is essentially contained in \cite{AaPa2a}. For convenience, we give the detailed proof in Lemma \ref{phb9bv} below.  Meanwhile, given $0<\eta\ll1$, for any $p,q\in (2,6)$ satisfying
\begin{align}\label{ASQ}
\frac{1}{p}+\frac{1}{q}=\frac{2}{3}+\eta,
\end{align}
and any $\theta\in (\frac{1}{2}-\frac{3}{2}\eta,\frac{1}{2}]$, we can divide $\theta$ into $\theta_1+\theta_2=\theta$, such that $\theta_1>\frac{3}{2}(\frac{1}{2}-\frac{1}{q})$, $\theta_2>\frac{3}{2}(\frac{1}{2}-\frac{1}{p})$. Thus by our claim (\ref{i987}),
\begin{align}\label{kiu909k}
\left| {T^{\theta_1,\theta_2}_{j,\infty}}(F,G) \right| \lesssim {j^{ - \infty }}{\left\| F \right\|_{L_t^2L_x^{p'}}}{\left\| G \right\|_{L_t^2L_x^{q'}}}.
\end{align}
For a given $\alpha>0$, choose $p$ slightly larger than $2$ such that
\begin{align}\label{AQ}
\alpha\frac{2p'}{2-p'}>1,
\end{align}
then H\"older and (\ref{kiu909k}) give
\begin{align}\label{kiu909}
\left| T^{\theta_1,\theta_2}_{j,\infty}(F,G) \right| &\lesssim {j^{ - \infty }}\|\rho^{\alpha}\|_{L^{\frac{2p'}{2-p'}}}{\left\| F\rho^{-\alpha} \right\|_{L_t^2L_x^2}}{\left\| G \right\|_{L_t^2L_x^{q'}}}\nonumber\\
&\lesssim {j^{ - \infty }}{\left\| F\rho^{-\alpha} \right\|_{L_t^2L_x^2}}{\left\| G \right\|_{L_t^2L_x^{q'}}}.
\end{align}
Notice that for $\theta_1+\theta_2=\theta\in(\frac{1}{2}-\eta,\frac{1}{2}]$, (\ref{AQ}) and (\ref{ASQ}) imply that in order to obtain (\ref{kiu909}), $q$ should be restricted to $q\in(\frac{6}{1+6\alpha+6\eta},6)$. The special case of (\ref{kiu909}) when $\theta_1+\theta_2=\frac{1}{2}$ corresponds to (\ref{6.22p}). The left range of $q$ will be considered later. \\

{\bf Step 2.3.2. Low Frequency for Positive  $j$.}
Meanwhile, (\ref{masu7}) and H\"older give for all $p,q\in(2,6)$ and any $\sigma_2\in\Bbb R$,
\begin{align}
&\left| {T_{j,0}^{0,\sigma_2}(F,G)}\right| \nonumber\\
&\lesssim \int_{\Bbb R} {\int_{t - {2^j} \le s \le t - {2^{j - 1}}}} {\left\|(I-\chi_{\infty}(D)) \widetilde{D}^{-\sigma_2}e^{\pm i(t-s)D}{F(s)} \right\|_{{L^{q}}}}{\left\| {G(s)} \right\|_{{L^{q'}}}}dsdt \nonumber\\
&\lesssim \int_{\Bbb R} {\int_{t - {2^j} \le s \le t - {2^{j - 1}}} {{{(t - s)}^{ - \frac{3}{2}}}} } {\left\| {F(s)} \right\|_{{L^{p'}}}}{\left\| {G(s)} \right\|_{{L^{q'}}}}dsdt \nonumber\\
&\lesssim {2^{ - j/2}}{\left\| {F(s)} \right\|_{L_t^2L_x^{p'}}}{\left\| {G(s)} \right\|_{L_t^2L_x^{q'}}}.\label{uyui}
\end{align}
Choosing $p$ to be slightly larger than 2 such that (\ref{AQ}) holds, we have by (\ref{uyui}) that
\begin{align}\label{kiu9090}
\left| {T^{0,\sigma_2}_{j,0}}(F,G) \right| \lesssim {2^{-j/2}}{\left\| \rho^{-\alpha}F \right\|_{L_t^2L_x^2}}{\left\| G \right\|_{L_t^2L_x^{q'}}}.
\end{align}
The special case when $\sigma_2=\frac{1}{2}$ corresponds to the low frequency part of (\ref{6.22p}). \\
{\bf Step 2.4. Sum for Positive $j$. }  Hence, (\ref{masu2}) is summable when $j\ge0$ by (\ref{kiu909}) and (\ref{uyui}) for $q\in(\frac{6}{1+6\alpha+6\eta},6)$. Thus we have proved (\ref{ky6tfcser}) for $q\in(\frac{6}{1+6\alpha},6)$. It remains to prove (\ref{ky6tfcser}) for the left $p$ in (2,6).  Since the negative $j$ part of (\ref{6.22p}) is done, we separate the positive $j$ part by defining:
\begin{align}
T^{\gamma}_{\ge 0}F:=\int^{t-\frac{1}{2}}_{-\infty}e^{\pm itD}D^{-\gamma}F(s)ds.
\end{align}
Then it suffices to prove
\begin{align}\label{hjyu}
\|T^{\frac{1}{2}}_{\ge 0}F\|_{L_t^2L_x^q} \lesssim {\left\| {{\rho ^{ - \alpha }}F} \right\|_{L_t^2L_x^2}}.
\end{align}
Denote the high frequency truncation of $T^{\theta}_{\ge 0}$ by $T^{\gamma}_{\ge 0,hi}$ and its low frequency truncation by $T^{\theta}_{\ge 0,low}$ respectively.
And the corresponding bilinear form can be divided into dyadic subintervals and high/low frequency parts as above. The only difference is $j$ is forced to be $j\ge0$. Step 2.3.2 shows the low frequency part $T^{\gamma}_{\ge 0,low}F$ satisfies
\begin{align}\label{meilig41}
{\left\| T^{\gamma}_{\ge 0,low}F\right\|_{L_t^2L_x^q}} \lesssim {\left\| {{\rho ^{ - \alpha }}F} \right\|_{L_t^2L_x^2}}.
\end{align}
for all $\gamma\in\Bbb R$ and $q\in(2,6)$.
Step 2.3.1 shows the high frequency part $T^{\gamma}_{\ge 0,hi}F$ satisfies
\begin{align}\label{meilig42}
{\left\| T^{\gamma}_{\ge 0,hi}F\right\|_{L_t^2L_x^q}} \lesssim {\left\| {{\rho ^{ - \alpha }}F} \right\|_{L_t^2L_x^2}}.
\end{align}
for all $\gamma\in (\frac{1}{2}-\eta,\frac{1}{2}]$ and $q\in(\frac{6}{1+6\alpha+6\eta},6)$. The key point is (\ref{meilig41}), (\ref{meilig42}) give some gain in derivatives. (But it seems that this gain only happens in the positive $j$ part.)

\noindent{\bf Step 2.4.1. Derivatives of Low Order. } Consider $T^{1}_{\ge 0}$. The corresponding dyadic bilinear form is
\begin{align}\label{1suiyue11}
{T^{1 ,0}_{j,\infty}}(F,G) = \int_{{\Bbb R}} {\int_{\Bbb D}} {\int_{t - {2^j} \le s \le t - {2^{j - 1}} }}{{e^{i(t - s)D}} } { {D}^{ - 1}}\chi_{\infty}F(s)G(t)dsdxdt,
\end{align}
and
\begin{align}\label{1suiyue23}
{T^{1 ,0}_{j,0}}(F,G) = \int_{{\Bbb R}} {\int_{\Bbb D}} {\int_{t - {2^j} \le s \le t - {2^{j - 1}} }}{{e^{i(t - s)D}} } { {D}^{ - 1}}(I-\chi_{\infty}(D))F(s)G(t)dsdxdt,
\end{align}
In this case, $T^{1 ,0}_{j,\infty}$ has a full derivative. Then by directly applying Corollary \ref{hasonggu}, we obtain for all $p,q\in(2,6)$
\begin{align*}
\left| {T_{j,\infty}^{1,0}(F,G)}\right| &\lesssim \int_{\Bbb R} {\int_{t - {2^j} \le s \le t - {2^{j - 1}}} {{{(t - s)}^{ -\infty}}} } {\left\| {F(s)} \right\|_{{L^{p'}}}}{\left\| {G(s)} \right\|_{{L^{q'}}}}dsdt \nonumber\\
&\lesssim {j^{ -\infty }}{\left\| {F(s)} \right\|_{L_t^2L_x^{p'}}}{\left\| {G(s)} \right\|_{L_t^2L_x^{q'}}}.
\end{align*}
Then by choosing $p$ to be slightly larger than 2, we get for all $q\in(2,6)$,
\begin{align}\label{niubi99n}
\left| {T^{0,1}_{j,\infty}}(F,G) \right| \lesssim {j^{-\infty }}{\left\| \rho^{-\alpha}F \right\|_{L_t^2L_x^2}}{\left\| G \right\|_{L_t^2L_x^{q'}}}.
\end{align}
The low frequency part $T^{1,0}_{j,0}$ in (\ref{1suiyue23}) for $j\ge0$ follows by the same arguments as $D^{\frac{1}{2}}$ considered above. Hence, we have for all $q\in(2,6)$
\begin{align}\label{meilig4}
{\left\| T^{1}_{\ge 0}\right\|_{L_t^2L_x^q}} \lesssim {\left\| {{\rho ^{ - \alpha }}F} \right\|_{L_t^2L_x^2}}.
\end{align}
{\bf Step 2.4.2. Full Range of $q$ by Interpolation.} It suffices to prove (\ref{hjyu}). By Gagliardo-Nirenberg inequality, for any $q\in(2,6)$ there exists $q_1=2^+$, $q_2=6^{-}$ and $\gamma=(\frac{1}{2})^{-}$ such that
\begin{align}\label{meilig4}
\left\| T^{\frac{1}{2}}_{\ge 0}F\right\|_{L_x^q} \lesssim {\left\| T^{1}_{\ge 0}F\right\|^{\tau}_{L_x^{q_1}} }{\left\| T^{\gamma}_{\ge 0}F\right\|^{1-\tau}_{L_x^{q_2}}},
\end{align}
with $\tau\in(0,1)$. Then (\ref{hjyu}) follows by (\ref{meilig4}), (\ref{meilig42}), (\ref{meilig41}) and H\"older.
\end{proof}

Proposition \ref{jiu9} with Lemma \ref{woshishui}, Lemma \ref{whishui} gives
\begin{lemma}\label{tianxia1}
Let $H$ satisfy assumptions in Theorem \ref{1}, $0<\alpha< 2\varrho$,  then we have the weighted Strichartz estimates for the magnetic wave equation: If $u$ solves the equation
\begin{align*}
\left\{ \begin{array}{l}
\partial _t^2u +Hu = 0 \\
u(0,x) = {u_0},{\partial _t}u(0,x) = {u_1} \\
\end{array} \right.
\end{align*}
then it holds for any $p\in(2,6)$
\begin{align*}
&{\left\| {{{D}^{\frac{1}{2}}}u} \right\|_{L_t^2L_x^{p}}} +{\left\| {{\rho ^{\alpha}}\nabla u} \right\|_{L_t^2L_x^2}}+
{\left\| {{\partial _t}u} \right\|_{L_t^\infty L_x^2}} + {\left\| {\nabla u} \right\|_{L_t^\infty L_x^2}}\\
&\lesssim {\left\| {\nabla {u_0}} \right\|_{{L^2}}} + {\left\| {{u_1}} \right\|_{{L^2}}}.
\end{align*}
\end{lemma}
\begin{proof}
(\ref{p9ij}), Lemma \ref{thing} and Lemma \ref{woshishui} with $p\in(2,6)$ give
\begin{align*}
 {\left\| {{{D}^{\frac{1}{2}}}u} \right\|_{L_t^2L^p_x}} + {\left\| u \right\|_{L_t^2L^p_x}} &\lesssim{\left\| {\nabla {u_0}} \right\|_{{L^2_x}}} + {\left\| {{u_1}} \right\|_{{L^2_x}}}+{\left\| {{\rho ^{ -\alpha }}Wu} \right\|_{L_t^2L^2_x}}.
\end{align*}
Hence by (\ref{Vab}), one obtains
\begin{align}\label{ufo}
 {\left\| {{{D}^{\frac{1}{2}}}u} \right\|_{L_t^2L^p_x}} + {\left\| u \right\|_{L_t^2L^p_x}}
\lesssim {\left\| {\nabla {u_0}} \right\|_{{L^2_x}}} + {\left\| {{u_1}} \right\|_{{L^2_x}}}+{\left\| {{\rho ^{\alpha} }\nabla u} \right\|_{L_t^2L^2_x}} +  {\left\| \rho ^{\alpha}u \right\|_{L_t^2L^2_x}}.
\end{align}
Lemma \ref{whishui} shows
\begin{align}\label{ufo2}
{\left\| {{\rho ^\alpha }\nabla u} \right\|_{L_t^2L^2_x}} \lesssim {\left\| {\nabla {u_0}} \right\|_{L^2_x}} + {\left\| {{u_1}} \right\|_{L^2_x}}.
\end{align}
By (\ref{ufo}), (\ref{ufo2}) and the Kato smoothing effect for $e^{\pm it\sqrt{H}}$, we get the endpoint homogeneous estimate
\begin{align*}
{\left\| {{{D}^{\frac{1}{2}}}u} \right\|_{L_t^2L^p_x}}  + {\left\| u \right\|_{L_t^2L^p_x}} \lesssim {\left\| {\nabla {u_0}} \right\|_{{L^2}}} + {\left\| {{u_1}} \right\|_{{L^2}}}.
\end{align*}
\end{proof}

The remaining inhomogeneous endpoint Strichartz estimates are given below.
\begin{proposition}\label{tianxia}
Let $H$ satisfy assumptions in Theorem \ref{1}, $0<\alpha<2\varrho$,  then we have the weighted Strichartz estimates for the magnetic wave equation: If $u$ solves the equation
\begin{align*}
\left\{ \begin{array}{l}
\partial _t^2u +Hu =F\\
u(0,x) = {u_0},{\partial _t}u(0,x) = {u_1} \\
\end{array} \right.
\end{align*}
then it holds for any $p\in(2,6)$
\begin{align*}
&{\left\| {{{D}^{\frac{1}{2}}}u} \right\|_{L_t^2L_x^{p}}} +{\left\| {{\rho ^{\alpha}}\nabla u} \right\|_{L_t^2L_x^2}}+
{\left\| {{\partial _t}u} \right\|_{L_t^\infty L_x^2}} + {\left\| {\nabla u} \right\|_{L_t^\infty L_x^2}}\\
&\lesssim {\left\| {\nabla {u_0}} \right\|_{{L^2}}} + {\left\| {{u_1}} \right\|_{{L^2}}} + {\left\| F \right\|_{L_t^1L_x^2}}.
\end{align*}
\end{proposition}
\begin{proof}
It remains to prove the case when $u_0=u_1=0$ by Lemma \ref{tianxia1}. In this case, by the Christ-Kiselev lemma, it suffices to prove
\begin{align*}
\left\|\int_{\Bbb R} D^{\frac{1}{2}}H^{-\frac{1}{2}}e^{\pm i (t-s)\sqrt{H}}F(s)ds \right\|_{L_t^2L_x^{p}}\lesssim {\left\| F \right\|_{L_t^1L_x^2}}.
\end{align*}
This follows immediately from Minkowski, Lemma \ref{tianxia1} and (\ref{equi1})-(\ref{equi21}).
\end{proof}

\section{Proof of Corollary 1.1}
In this section we prove Corollary 1.1. The Kato smoothing effect will be proved first, the equivalence of $H^{\frac{1}{2}}$ and $D$ in various spaces will be proved then.
\subsection{Kato smoothing estimates for wave equations with small potentials}
In this subsection, we aim to prove the Kato smoothing estimates for the magnetic half wave operator $e^{it\sqrt{H}}$. The technique we use is on one hand quite eclectic, and usually statements are proved by combining ideas borrowed from different works. And on the other hand, some refinements and new ideas are introduced to estimate troublesome terms.

The self-adjoint operator $H$ is strictly positive due to the smallness assumption of potentials. In fact, (\ref{Vab1}) shows
\begin{align}
\langle Hf,f\rangle=\langle-\Delta f,f\rangle+O(\mu_1\|\nabla f\|^2_{L^2}+\mu_1\|f\|^2_{L^2}).
\end{align}
Then the standard Poincare inequality  $\langle-\Delta f,f\rangle\ge \frac{1}{4}\|f\|^2_{L^2}$ implies there exists some $c>0$ such that
$\langle Hf,f\rangle\ge c\|f\|^2_{L^2}$ provided $\mu_1$ is sufficiently small.

We recall the Kato smoothing Theorem.
\begin{theorem}[\cite{Mathphysics}](Kato smoothing theorem)\label{p6oijhn}
Let $M,N$ be two Hilbert spaces and $H:M\to N$ be a self-adjoint operator. Denote its resolvent by $(H-\lambda)^{-1}$. Let $U:M\to N$ be a closed densely defined operator. Assume that for any $f\in D(U^*)$, $\lambda\in\Bbb C\setminus\Bbb R$  there holds
\begin{align*}
\|U(H-\lambda)^{-1}U^*f\|_{N}\le C\|f\|_{N}.
\end{align*}
Then $e^{\pm itH}g\in D(U)$ for all $g\in M$ and a.e. $t$. Moreover, it holds
\begin{align*}
\int^{\infty}_{-\infty}\|Ue^{\pm itH}g\|^2_{N}dt\le \frac{2}{\pi}C^2\|g\|^2_{M}.
\end{align*}
\end{theorem}

First we give the estimates for the kernel of the free resolvent. (\ref{serr}), (\ref{sedr4r}) were established by Corollary 3.2 and Lemma 3.3 in \cite{BaMa}. (\ref{sedr564r}) and (\ref{jiujinfa4}) are new here.
\begin{lemma}\label{s564rt67}
Let $\mathcal{R}_0(\frac{1}{2}+\sigma)=(-\Delta+\sigma^2-\frac{1}{4})^{-1}$ be the free resolvent and denote its Schwartz kernel by $\mathcal{R}_0(\frac{1}{2}+\sigma,x,y)$.
Then for $\Re \sigma\ge0$, $|\sigma|\le 1$, $r\in (0,\infty)$, we have
\begin{align}\label{sedr4r}
| {{\mathcal{R}_0}(\frac{1}{2} + \sigma ,x,y)}| \le \left\{ \begin{array}{l}
C| {\log r}|,\mbox{  }\mbox{  }| r| \le 1 \\
C{| \sigma|^{ - \frac{1}{2}}}{e^{ - (\frac{1}{2} + {\Re}\sigma )r}},| r| \ge 1 \\
\end{array} \right.
\end{align}
And for $\Re \sigma\ge0$, $|\sigma|\ge 1$, $r\in (0,\infty)$, we have
\begin{align}\label{serr}
| {{\mathcal{R}_0}(\frac{1}{2} + \sigma ,x,y)} | \le \left\{ \begin{array}{l}
 C| {\log r}|,\mbox{  }\mbox{  }| {r\sigma } | \le 1 \\
 C{| \sigma |^{ - \frac{1}{2}}}{e^{ - (\frac{1}{2} + {\Re} \sigma )r}},| {r\sigma }| \ge 1 \\
 \end{array} \right.
\end{align}
Furthermore, for $\Re \sigma\ge0$, $|\sigma|\le 1$, $r\in (0,\infty)$, we have
\begin{align}\label{jiujinfa4}
\left| {{\nabla _x}{\mathcal{R}_0}(\frac{1}{2} + \sigma ,x,y)} \right| \le \left\{ \begin{array}{l}
 C{r^{ - 2}}{( {\sinh r})^2}( {{{\cosh }^2}r - 1} )^{ - \frac{1}{2}},\mbox{  }| r| \le 1 \\
 C{| \sigma |^{\frac{1}{2}}}{e^{ - (\frac{3}{2} + {\Re}\sigma )r}}{( {\sinh r})^2}{( {{{\cosh }^2}r - 1})^{ - \frac{1}{2}}},| r | \ge 1 \\
\end{array} \right.
\end{align}
And for $\Re \sigma\ge0$, $|\sigma|\ge 1$, $r\in (0,\infty)$, we have
\begin{align}\label{sedr564r}
\left| {{\nabla _x}{\mathcal{R}_0}(\frac{1}{2} + \sigma ,x,y)} \right| \le \left\{ \begin{array}{l}
 C{r^{ - 2}}{\left( {\sinh r} \right)^2}{\left( {{{\cosh }^2}r - 1} \right)^{ - \frac{1}{2}}},\mbox{  }\mbox{  }\left| {r\sigma } \right| \le 1 \\
 C{\left| \sigma  \right|^{\frac{1}{2}}}{e^{ - (\frac{3}{2} + {\Re}\sigma )r}}{\left( {\sinh r} \right)^2}{\left( {{{\cosh }^2}r - 1} \right)^{ - \frac{1}{2}}},| {r\sigma }| \ge 1 \\
 \end{array} \right.
\end{align}

\end{lemma}
\begin{proof}
Let $[{}^n\mathcal{R}]_0(s,x,y)$ denote the Schwartz kernel of the resolvent of the Laplace-Beltrami operator in $\Bbb H^{n+1}$, i.e., the kernel of  $(-\Delta_{\Bbb H^{n+1}}-s(n-s))^{-1}$. Then $[{ }^n{\mathcal R}]_0(s,x,y)$ can be written in terms of Legendre functions (\cite{Taylor}),
\begin{align}\label{ju6cfder1}
[{ }^n\mathcal{R}]_0(s,x,y)=c(n)e^{-i\pi\mu}(\sinh r)^{-\mu}Q^{\mu}_{\nu}(\cosh r),
\end{align}
where $r=d(x,y)$, $\nu=s-\frac{n+1}{2}$ and $\mu=\frac{n-1}{2}$.
Particularly, if choosing $s=\frac{n}{2}+\sigma$ for $n=1,3$ we have
\begin{align}
[{ }^1\mathcal{R}]_0(\frac{1}{2}+\sigma,x,y)&=\mathcal{R}_0(\frac{1}{2}+\sigma,x,y)\label{ju46cfder}\\
[{ }^1\mathcal{R}]_0(\frac{1}{2}+\sigma,x,y)&=c(1)Q^{0}_{\sigma-\frac{1}{2}}(\cosh r),\label{ju6cfder}\\
[{ }^3\mathcal{R}]_0(\frac{1}{2}+\sigma,x,y)&=c(3)e^{-i\pi}(\sinh r)^{-1}Q^{1}_{\sigma-\frac{1}{2}}(\cosh r),\label{ju6cfd1er}
\end{align}
Recall the formula for the derivative of the second class Legendre functions (see for instance \cite{Watson})
\begin{align}\label{hsderuag87}
{( z^2 - 1)^{\frac{1}{2}}}\frac{d}{dz}{Q^{0}_\nu}(z) = Q_\nu^1(z).
\end{align}
Therefore, (\ref{ju46cfder}) combined with (\ref{hsderuag87}),  (\ref{ju6cfder}) and (\ref{ju6cfd1er}) implies
\begin{align*}
 {\partial _r}{\mathcal{R}_0}(\frac{1}{2} + \sigma ,x,y) &= c(1){\partial _r}{Q_\kappa }(\cosh r) = c(1)( {\cosh^2}r - 1)^{ - \frac{1}{2}}\sinh rQ_\kappa ^1(\cosh r) \\
 &= c({\cosh^2}r - 1 )^{ - \frac{1}{2}}{ {\sinh^2 r} }[{}^3\mathcal{R}]_0(\frac{3}{2} + \sigma ,x,y).
\end{align*}
Hence, (\ref{sedr564r}) and (\ref{jiujinfa4}) follow from the fact $|\nabla r|=1$ and the corresponding resolvent estimates for $[{ }^3\mathcal{R}]_0$ in [Corollary 3.2, Lemma 3.3,\cite{BaMa}].
\end{proof}

The free resolvent $\mathcal{R}_0(z)=(-\Delta-z(1-z))^{-1}$ has the following basic estimates in weighted $L^2$ spaces.
\begin{lemma}\label{xianren}
For $z\in\Bbb C$ with $\Re z>0$, we have for $\alpha>0$ sufficiently small
\begin{align}
\|\rho^{\alpha}\mathcal{R}_{0}(\frac{1}{2}+z)f\|_{L^2}&\lesssim|z|^{-1}\|f\rho^{-\alpha}\|_{L^2}\label{huashen1}\\
\|\rho^{\alpha}\mathcal{R}_{0}(\frac{1}{2}+z)f\|_{L^2}&\lesssim\|f\rho^{-\alpha}\|_{L^2}\label{huashen2}\\
\|\rho^{\alpha}\nabla \mathcal{R}_{0}(\frac{1}{2}+z)f\|_{L^2}&\lesssim\|f\rho^{-\alpha}\|_{L^2}.\label{huashen3}
\end{align}
\end{lemma}
\begin{proof}
(\ref{huashen1}) follows by the same arguments in [\cite{BaMa}, Proposition 4.1] and the following wave operator expression for the free resolvent:
\begin{align}
(-\Delta-\frac{1}{4}-(\lambda+i\mu)^2)^{-1}=\Lambda(\lambda,\mu)\int^{\infty}_0e^{i({\rm {sgn}} \mu)\lambda t}e^{-|\mu|t}(\cos t\sqrt{-\Delta-\frac{1}{4}})dt,
\end{align}
where $\Lambda(\lambda,\mu)=\frac{i{{\rm{sgn}}\mu}}{(\lambda+i\mu)}$.
By (\ref{sedr564r}), (\ref{jiujinfa4}), the kernel of  $\rho^{\alpha}\nabla \mathcal{R}_{0}\rho^{\alpha}$ is bounded by
$$
\left\{ \begin{array}{l}
 C{e^{ - (\alpha {+}\frac{1}{2} )r}}{r^{ - \frac{1}{2}}},\mbox{  }\left| z \right| \le 1 \\
 C{r^{ - \frac{1}{2}}},\mbox{  }\left| z \right| \ge 1,\left| {zr} \right| \le 1 \\
 C{\left| z \right|^{\frac{1}{2}}}{e^{ - (\alpha {+}\frac{1}{2} )r}}{r^{ - \frac{1}{2}}},\mbox{  }\left| z \right| \ge 1,\left| {zr} \right| \ge 1 \\
\end{array} \right.
$$
Then, for $|z|\le1$ the Kunze-Stein phenomenon yields
\begin{align}\label{p65gvp}
\|\rho^{\alpha}\nabla\mathcal{R}_0(\frac{1}{2}+z)f\|_{L^2}\lesssim \|\rho^{\alpha}f\|_{L^2},
\end{align}
which shows (\ref{huashen3}) in the case $|z|\le1$.
By Sobolev embedding and Leibnitz rule,
\begin{align}
{\left\| {{\rho ^\alpha }{{\cal R}_0}(\frac{1}{2}  + z){\rho ^\alpha }f} \right\|_{{L^2}}} &\le C{\left\| {\nabla {\rho ^\alpha }{{\cal R}_0}(\frac{1}{2}  + z){\rho ^\alpha }f} \right\|_{{L^2}}} \label{poingy8}\\
&\le C\alpha {\left\| {{\rho ^\alpha }{{\cal R}_0}(\frac{1}{2}  + z){\rho ^\alpha }f} \right\|_{{L^2}}} + C{\left\| {{\rho ^\alpha }\nabla {{\cal R}_0}(\frac{1}{2}  + z){\rho ^\alpha }f} \right\|_{{L^2}}}.\nonumber
\end{align}
Since (\ref{huashen1}) has shown the LHS of (\ref{poingy8}) is finite provided $\Re z>0$, choosing $\alpha>0$ to be sufficiently small, we get
\begin{align*}
{\left\| {{\rho ^\alpha }{{\cal R}_0}(\frac{1}{2}  + z){\rho ^\alpha }f} \right\|_{{L^2}}} \le C{\left\| {{\rho ^\alpha }\nabla {{\cal R}_0}(\frac{1}{2}  + z){\rho ^\alpha }f} \right\|_{{L^2}}}.
\end{align*}
Hence (\ref{huashen2}) follows by  (\ref{huashen1}) when $|z|\ge 1$ and by (\ref{p65gvp}) when $|z|\le1$. The rest is to prove (\ref{huashen3}) for $|z|\ge1$. Integration by parts and $|\nabla d(x,0)|=1$ yield
\begin{align*}
&\left\| {{\rho ^\alpha }\nabla \mathcal{R}_0(\frac{1}{2}  + z){\rho ^\alpha }f} \right\|_{{L^2}}^2
= \left\langle {{\rho ^\alpha }\nabla \mathcal{R}_0(\frac{1}{2}  + z){\rho ^\alpha }f,{\rho ^\alpha }\nabla \mathcal{R}_0(\frac{1}{2}  + z){\rho ^\alpha }f} \right\rangle  \\
&=  - \left\langle {{\rho ^{2\alpha }}\Delta \mathcal{R}_0(\frac{1}{2}  + z){\rho ^\alpha }f,\mathcal{R}_0(\frac{1}{2}  + z){\rho ^\alpha }f} \right\rangle
\\
&\mbox{  }\mbox{  }\mbox{  }- \left\langle {\left( {\nabla {\rho ^{2\alpha }}} \right)\cdot\nabla \mathcal{R}_0(\frac{1}{2}  + z){\rho ^\alpha }f,\mathcal{R}_0(\frac{1}{2}  + z){\rho ^\alpha }f} \right\rangle  \\
&= \left\langle {{\rho ^{2\alpha }}\left( { - \Delta  - {\frac{1}{4}} + {z^2}} \right)\mathcal{R}_0(\frac{1}{2}  + z){\rho ^\alpha }f,\mathcal{R}_0(\frac{1}{2}  + z){\rho ^\alpha }f} \right\rangle \\
&\mbox{  }\mbox{  }+ O\left( {\alpha \left\langle {{\rho ^\alpha }\nabla \mathcal{R}_0(\frac{1}{2}  + z){\rho ^\alpha }f,{\rho ^\alpha }\mathcal{R}_0(\frac{1}{2}  + z){\rho ^\alpha }f} \right\rangle } \right) \\
&\mbox{  }\mbox{  }+ \left( {{\frac{1}{4}} - {z^2}} \right)\left\langle {{\rho ^\alpha }\mathcal{R}_0(\frac{1}{2}  + z){\rho ^\alpha }f,{\rho ^\alpha }\mathcal{R}_0(\frac{1}{2}  + z){\rho ^\alpha }f} \right\rangle.
\end{align*}
Thus  for $|z|\ge 1$, (\ref{huashen1}) gives
\begin{align*}
 &\left\| {{\rho ^\alpha }\nabla \mathcal{R}_0(\frac{1}{2}  + z){\rho ^\alpha }f} \right\|_{{L^2}}^2\\
 &\lesssim {\left\| f \right\|_{{L^2}}}{\left\| {{\rho ^\alpha }\mathcal{R}_0(\frac{1}{2}  + z){\rho ^\alpha }f} \right\|_{{L^2}}} + \alpha \left\| {{\rho ^\alpha }\nabla \mathcal{R}_0(\frac{1}{2}  + z){\rho ^\alpha }f} \right\|_{{L^2}}^2 \\
 &+ \left( {{{\left| z \right|}^2} + 1} \right)\left\| {{\rho ^\alpha }\mathcal{R}_0(\frac{1}{2}  + z){\rho ^\alpha }f} \right\|_{{L^2}}^2\\
 &\lesssim \alpha \left\| {{\rho ^\alpha }\nabla \mathcal{R}_0(\frac{1}{2}  + z){\rho ^\alpha }f} \right\|_{{L^2}}^2+\|f\|^2_{L^2}.
\end{align*}
Let $\alpha>0$ be sufficiently small, we obtain for $|z|\ge1$
\begin{align*}
\left\| {{\rho ^\alpha }\nabla \mathcal{R}_0(\frac{1}{2}  + z){\rho ^\alpha }f} \right\|_{{L^2}}\lesssim \|f\|_{L^2},
\end{align*}
which combined with (\ref{p65gvp}) gives (\ref{huashen3}).
\end{proof}

\begin{lemma}
Recall $\mathcal{R}_H(z)=(H-z(1-z))^{-1}$, then for all $\Re z>0$, one has when $0<\mu_1\ll 1,0<\alpha\ll1$
\begin{align}\label{chun1}
\|\rho^{\alpha}\mathcal{R}_H(z+\frac{1}{2})f\|_{L^2}\le C\min(1,|z|^{-1})\|\rho^{-\alpha}f\|_{L^2}.
\end{align}
\end{lemma}
\begin{proof}
Formally we have the identity
\begin{align}
\mathcal{R}_H(z+\frac{1}{2})=\mathcal{R}_0(z+\frac{1}{2})\big(I+W\mathcal{R}_0(z+\frac{1}{2})\big)^{-1}.
\end{align}
First we prove the operator $(I+W\mathcal{R}_0)^{-1}$ is well-defined and uniformly bounded in $L(\rho ^{\alpha}L^2,\rho^{\alpha}L^2)$.
By Lemma \ref{xianren}, choose $3\alpha<\varrho$,
\begin{align*}
{\left\| {{\rho ^{ - \alpha }}X{\mathcal{R}_0}f} \right\|_{{L^2}}}& \lesssim \left\| |A|{\rho ^{ - 2\alpha }} \right\|_{{L^\infty }}{\left\| {{\rho ^\alpha }\nabla{\mathcal{R}_0}f} \right\|_{{L^2}}}\\
&\lesssim \mu_1{\left\| {{\rho ^\alpha }\nabla{\mathcal{R}_0}{\rho ^\alpha }} \right\|_{{L^2} \to {L^2}}}{\left\| {{\rho ^{ - \alpha }}f} \right\|_{{L^2}}}.
\end{align*}
Meanwhile, we have
$$
{\left\| {{\rho ^{ - \alpha }}V{\mathcal{R}_0}f} \right\|_{{L^2}}} \lesssim \mu_1 {\left\| {{\rho ^\alpha }{\mathcal{R}_0}{\rho ^\alpha }} \right\|_{{L^2} \to {L^2}}}{\left\| {{\rho ^{ - \alpha }}f} \right\|_{{L^2}}}.
$$
Therefore, by Neumann series argument we conclude $I+W\mathcal{R}_0$ is invertible in $\rho^{\alpha}L^2$ with
\begin{align*}
{\left\| {{\rho ^{ - \alpha }}(I + W{\mathcal{R}_0})^{-1}f} \right\|_{{L^2}}} \le C{\left\| {{\rho ^{ - \alpha }}f} \right\|_{{L^2}}},
\end{align*}
where $C$ is independent of $z$.
Hence Lemma \ref{xianren} implies
\begin{align*}
{\left\| {{\rho ^\alpha }{\mathcal{R}_0}{{(I + W{\mathcal{R}_0})}^{ - 1}}f} \right\|_{{L^2}}} &\le C{\left\| {{\rho ^\alpha }{\mathcal{R}_0}{\rho ^\alpha }} \right\|_{{L^2} \to {L^2}}}{\left\| {{\rho ^{ - \alpha }}{{(I + W{\mathcal{R}_0})}^{ - 1}}f} \right\|_{{L^2}}}\\
&\lesssim{\left\| {{\rho ^{ - \alpha }}f} \right\|_{{L^2}}}.
\end{align*}
The $|z|^{-1}$ decay follows directly from (\ref{huashen1}).
\end{proof}

\begin{corollary}\label{tudiye}
Let $H$ satisfy the assumptions in Corollary \ref{3}.
If  $0<\mu_1\ll 1$, then the spectrum of $H$ is absolutely continuous and $\sigma(H)=[\frac{1}{4},\infty)$.
\end{corollary}
\begin{proof}
In the beginning  of Section 2, we have shown $H$ is self-adjoint. By [Theorem XIII.19, Page 137,\cite{Mathphysics}] and the continuity of spectral projection operators, to prove the spectrum is absolutely continuous, it suffices to prove for any bounded interval $(c,d)$ and any $f\in C^{\infty}_c$
\begin{align}\label{pdoi890}
\mathop {\sup }\limits_{0 < \varepsilon  < 1}\int_c^d\left| \Im\left\langle {f,R_{H}(\tau + i\varepsilon )f} \right\rangle \right|^2d\tau < \infty.
\end{align}
Using (\ref{chun1}) one has for $\varepsilon>0$
$$\left| {\left\langle {f ,R_{H}(\tau + i\varepsilon )f } \right\rangle } \right| = \left| {\left\langle {f {\rho ^{ - \alpha }},{\rho ^\alpha }R_H(\tau + i\varepsilon ){\rho ^\alpha }{\rho ^{ - \alpha }}f} \right\rangle } \right| \lesssim \left\| {f {\rho ^{ - \alpha }}} \right\|_2^2 < \infty,
$$
which yields (\ref{pdoi890}).
Meanwhile, by Weyl's criterion, $\sigma_{ess}(H)=[\frac{1}{4},\infty)$. Therefore, we obtain $\sigma(H)=\sigma_{ac}(H)=[\frac{1}{4},\infty)$.
\end{proof}

\begin{lemma}\label{wqcun9}
Let $H$ satisfy the assumptions in Corollary \ref{3}.
Let $z\in\Bbb C\backslash\Bbb R$. If $0<\mu_1\ll 1,0<\alpha\ll1$, then there holds
\begin{align}\label{cun9}
{\left\| {{\rho ^\alpha }R_{\sqrt{H}}(z)f} \right\|_{{L^2_x}}} \le C{\left\| {{\rho ^{ - \alpha }}f} \right\|_{{L^2_x}}},
\end{align}
where $C$ is independent of $z$. And thus by Theorem \ref{p6oijhn}, the Kato smoothing effect
$$\|\rho^{\alpha}e^{\pm i\sqrt{H}t}f\|_{L^2_{t,x}}\lesssim \|f\|_{L^2_x}
$$
holds for all $0<\alpha\ll1$.
\end{lemma}
\begin{proof}
The proof is analogous to Lemma 2.7 of \cite{DG2}.
Recall $R_{\sqrt{H}}(z)=(\sqrt{H}-z)^{-1}$.
We consider two cases: $(a)$ $\Re z\le0$, $(b)$ $\Re z>0$. In case $(a)$, by Corollary \ref{tudiye} and spectrum theorem of self-adjoint operators, one has
\begin{align*}
 \left\| {( {\sqrt{H}} - z)f} \right\|_{{L^2}}^2 &= \left\| { {\sqrt{H}} f} \right\|_{{L^2}}^2 + {\left| z \right|^2}\left\| f \right\|_{{L^2}}^2 - 2{{\Re}} z\left\langle { {\sqrt{H}} f,f} \right\rangle  \\
 &\ge \left\| { {\sqrt{H}} f} \right\|_{{L^2}}^2 \ge \frac{1}{4} \left\| f \right\|_{{L^2}}^2.
\end{align*}
Thus we have
$${\left\| {{\rho ^\alpha }{{(\sqrt{H} - z )}^{ - 1}}f} \right\|_{{L^2}}} \le {\left\| {{{(\sqrt{H} -z )}^{ - 1}}f} \right\|_{{L^2}}} \le C{\left\| f \right\|_{{L^2}}} \le C{\left\| {{\rho ^{ - \alpha }}f} \right\|_{{L^2}}}.
$$
Hence $(a)$ is done.
For $(b)$, we use
\begin{align}\label{suiyueyou}
R_{\sqrt{H}}(z)=(\sqrt{H}-z)^{-1}=2z R_H(z^2)+(\sqrt{H}+z)^{-1}.
\end{align}
Since $\Re z>0$ and $z\notin \Bbb R$ in case $(b)$, we have $z^2\notin [\frac{1}{4},\infty)$. Then the first term in (\ref{suiyueyou}) follows from (\ref{chun1}) and the second follows from case $(a)$.
\end{proof}

In the following lemmas, we prove the equivalence of $\|(-\Delta)^sf\|_{\rho^{-\beta}L^2}$ and $\|H^sf\|_{\rho^{-\beta}L^2}$ for $s=\frac{1}{2}$ and $\beta=0,\alpha$.
As a preparation, we prove the $L^p$-$L^q$ estimates for the free resolvent.
\begin{lemma}\label{125dcferey}
Let $0<\alpha<\frac{1}{2}$.
For any $2\le p<q<\infty$, $\sigma\ge \frac{1}{2}$ and any $\omega>\frac{1}{p}-\frac{1}{q}$,  we have
\begin{align}
\|( - \Delta+\sigma ^2- {1/4} )^{ - 1}\|_{{L^{p}} \to {L^q}} &\lesssim \min(1,\sigma^{-2+\omega})\label{yvui}\\
\| \nabla( -\Delta+\sigma ^2- {1/4})^{ - 1} \|_{{L^p} \to {L^q}} &\lesssim \min(1,\sigma^{-1+2\omega})\label{yvui3}\\
\| ( -\Delta+\sigma ^2- {1/4})^{ - 1} \|_{\rho^{-\alpha}L^2 \to \rho^{-\alpha}L^2} &\lesssim \min(1,\sigma^{-2})\label{lcize34}\\
\| \nabla( -\Delta+\sigma ^2- {1/4})^{ - 1} \|_{\rho^{-\alpha}L^2 \to \rho^{-\alpha}L^2} &\lesssim \min(1,\sigma^{-1}).\label{lcize35}
\end{align}
\end{lemma}
\begin{proof}
First, we prove (\ref{yvui}) for $\frac{1}{2}\le \sigma\le 1$. By Lemma \ref{s564rt67} and Young's inequality, it suffices to prove for $\frac{1}{2}\le \sigma\le 1$ and $m\in(1,\infty)$ with $1+\frac{1}{q}=\frac{1}{p}+\frac{1}{m}$, there holds
 \begin{align*}
\int^{1}_0 |{\log r}|^{m}rdr\lesssim 1; \mbox{ }\sigma^{ - \frac{1}{2}}\int^{\infty}_1 {e^{ -m (\frac{1}{2} + \sigma )r}}e^{r}dr\lesssim 1.
\end{align*}
The first inequality on the LHS is obvious. Noticing that $m=1/(1-\frac{1}{p}+\frac{1}{q})$ is strictly larger than 1 due to $p<q$, for $\sigma\ge\frac{1}{2}$, we obtain
 \begin{align*}
\sigma^{ - \frac{1}{2}}\int^{\infty}_1 {e^{ -m (\frac{1}{2} +\sigma )r}}e^{r}dr\lesssim \frac{1}{m-1}.
\end{align*}
Second, we prove (\ref{yvui}) for $\sigma\ge 1$. By Lemma \ref{s564rt67} and Young's inequality, it suffices to prove for $\sigma\ge1$ there holds
\begin{align}
(\int^{\frac{1}{\sigma}}_0 |\log r|^m rdr)^{\frac{1}{m}}&\lesssim \sigma^{-2+\omega}\label{cvfp978}\\
(\int^{\infty}_{\sigma^{-1}} \sigma^{ - \frac{m}{2}}{e^{ - m(\frac{1}{2} +\sigma )r}}\sinh r dr)^{\frac{1}{m}}&\lesssim \sigma^{-2+\omega}.\label{cvfp97j8}
\end{align}
(\ref{cvfp978}) follows by direct calculation. For (\ref{cvfp97j8}), we divide it into two regimes: for $\sigma\ge1$, $r\in[\sigma^{-1},1]$, one has
\begin{align}
&(\int^{1}_{\sigma^{-1}} \sigma^{ - \frac{m}{2}}{e^{ - m(\frac{1}{2} +\sigma )r}}\sinh r dr)^{\frac{1}{m}}\lesssim (\int^{1}_{{\sigma}^{-1}} \sigma^{-\frac{m}{2}}{e^{ -\sigma r}}r dr)^{\frac{1}{m}}\\
&\lesssim (\int^{\sigma}_{1} \sigma^{-\frac{m}{2}-2}{e^{ -\tau}}\tau d\tau)^{\frac{1}{m}}\lesssim \sigma^{-2+\omega},
\end{align}
and for $\sigma\ge1$, $r\in[1,\infty)$ we have
\begin{align}
&(\int^{\infty}_{1} \sigma^{ - \frac{m}{2}}{e^{ - m(\frac{1}{2} +\sigma )r}}\sinh r dr)^{\frac{1}{m}}\lesssim (\int^{\infty}_{1} \sigma^{ - \frac{m}{2}}{e^{ (1- m) r}} e^{-(\sigma-\frac{1}{2})r}dr)^{\frac{1}{m}}\\
&\lesssim \left(e^{-(\sigma-\frac{1}{2})}\sigma^{ - \frac{m}{2}}\frac{1}{m-1}\right)^{\frac{1}{m}}\lesssim \sigma^{-2+\omega}.
\end{align}
Third, we prove (\ref{yvui3}) when $\frac{1}{2}\le\sigma\le 1$. This follows by the same arguments as above and the following inequalities for the corresponding kernel:
\begin{align*}
\sigma^{\frac{1}{2}}{r^{ - 2}}{e^{ - (\frac{3}{2} + \sigma )r}}{\left( {\sinh r} \right)^2}{\left( {{{\cosh }^2}r - 1} \right)^{ - \frac{1}{2}}}
\lesssim\left\{\begin{array}{ll}
                         r^{-\frac{1}{2}}, & \hbox{ } 0\le r  \le 1;\\
                         e^{(-\frac{1}{2}+\sigma)r}, & \hbox{ } r  \ge 1.
                       \end{array}
                     \right.
\end{align*}
Forth, we prove (\ref{yvui3}) when $\sigma\ge 1$. By Lemma \ref{s564rt67}, it suffices to prove
\begin{align}
&\sigma^{\frac{1}{2}}(\int^{1}_{\sigma^{-1}}{r^{ - 2m}}{\left( {\sinh r} \right)^{2m}}{\left( {{{\cosh }^2}r - 1} \right)^{ - \frac{m}{2}}}{e^{ - m(\frac{3}{2} + \sigma )r}}rdr)^{\frac{1}{m}}
\lesssim \sigma^{-1+2\omega}\label{cvdpjui8}\\
&\sigma^{\frac{1}{2}}(\int^{\infty}_{1}{r^{ - 2m}}{\left( {\sinh r} \right)^{2m}}{\left( {{{\cosh }^2}r - 1} \right)^{ - \frac{m}{2}}}{e^{ - m(\frac{3}{2} + \sigma )r}}e^{r}dr)^{\frac{1}{m}}\label{cdepjui8}
\lesssim \sigma^{-1+2\omega}.
\end{align}
When $\sigma\ge 1$, the LHS of (\ref{cvdpjui8}) is bounded by
\begin{align*}
\sigma^{\frac{1}{2}}(\int^{1}_{\sigma^{-1}}r^{ - \frac{m}{2}}re^{-m(\frac{3}{2} + \sigma )r}dr)^{\frac{1}{m}}
\lesssim \sigma^{\frac{1}{2}}(\sigma^{-2+\frac{m}{2}}\int^{\sigma}_{1}\tau^{ - \frac{m}{2}}\tau e^{-m\tau}d\tau)^{\frac{1}{m}}
\lesssim \sigma^{-1+2\omega},
\end{align*}
which yields (\ref{cvdpjui8}) .
When $\sigma\ge 1$, the LHS of (\ref{cdepjui8}) is bounded by
\begin{align*}
&\sigma^{\frac{1}{2}}(\int^{\infty}_{1}e^{[(m+1)-m(\frac{3}{2} + \sigma )]r}dr)^{\frac{1}{m}}
\lesssim e^{\sigma-\frac{1}{2}}\sigma^{\frac{1}{2}}(m-1)^{-\frac{1}{m}}
\lesssim \sigma^{-1+2\omega},
\end{align*}
which yields (\ref{cdepjui8}).\\
Fifth, we prove (\ref{lcize34}). Instead of Young's inequality, we will use the the Kunze-Stein phenomenon:
\begin{align}\label{phenomenon}
{\| {f * k}\|_{{L^2}({\Bbb D})}} \lesssim {\left\| f \right\|_{{L^2}({\Bbb D})}}\left\{ \int_0^\infty k(r){\varphi _0}(r)\sinh rdr  \right\},
\end{align}
where $k(r)$ is assumed to be radial.
Lemma \ref{s564rt67} shows the point-wise estimate:
\begin{align}
 \left|\rho^{\alpha d(x,O)}( -\Delta+\sigma ^2- {1/4})^{ - 1}(\rho^{-\alpha r}f)\right|\lesssim
 \int_{\Bbb D}e^{\alpha d(x,y)}\left|\mathcal{R}_0(x,y)f(y)\right|dy,\label{psde9ih}
\end{align}
and
\begin{align}
 \left|\rho^{\alpha d(x,O)}\nabla( -\Delta+\sigma ^2- {1/4})^{ - 1}(\rho^{-\alpha r}f)\right|\lesssim
 \int_{\Bbb D}e^{\alpha d(x,y)}\left|\nabla_x{\mathcal{R}_0}(x,y)f(y)\right|dy.\label{psde9ihh}
\end{align}
Inserting the bounds in Lemma \ref{s564rt67} to (\ref{psde9ih}), (\ref{psde9ihh}) and applying (\ref{phenomenon})  implies it suffices to prove
for $\sigma\in[\frac{1}{2},1]$
\begin{align}
&\sigma^{-\frac{1}{2}}\int^{1}_0 |{\log r}|r\varphi_0dr+\sigma^{-\frac{1}{2}}\int^{\infty}_1 e^{ -(\frac{1}{2} + \sigma )r}e^{(1+\alpha)r}\varphi_0dr\lesssim 1\label{afcvjfp97j8}\\
&\sigma^{\frac{1}{2}}\int^{\infty}_0{r^{- 2}} (\sinh r)^{2}( \cosh^2r - 1 )^{ -\frac{1}{2}}e^{ (\alpha- \frac{3}{2}- \sigma )r}\sinh r\varphi_0dr
\lesssim 1,\label{fbvjfp97j8}
\end{align}
and for $\sigma\in[1,\infty)$
\begin{align}
&\sigma^{-\frac{1}{2}}\int^{\frac{1}{\sigma}}_0 |\log r|e^{\alpha r}\varphi_0rdr+\sigma^{-\frac{1}{2}}\int^{\infty}_{\sigma^{-1}} \sigma^{- \frac{1}{2}}e^{(\alpha-\frac{1}{2} +\sigma )r}\sinh r\varphi_0 dr\lesssim \sigma^{-2}\label{fcvjfp97j8}\\
&\sigma^{\frac{1}{2}}\int^{1}_{\sigma^{-1}}r^{ - 2}({\sinh r})^{2}\left( \cosh^2r - 1 \right)^{ - \frac{1}{2}}e^{ (\alpha -\frac{3}{2} - \sigma )r}r\varphi_0dr\lesssim \sigma^{-1}\label{fcvdpjjui8}\\
&\sigma^{\frac{1}{2}}\int^{\infty}_{1}{r^{- 2}} (\sinh r)^{2}( \cosh^2r - 1 )^{ -\frac{1}{2}}e^{ (\alpha- \frac{3}{2}- \sigma )r}e^{r}\varphi_0dr\lesssim \sigma^{-1}.\label{cfdsjui8}
\end{align}
Using the bound $\varphi_0(r)\lesssim (1+r)e^{-\frac{1}{2} r}$ to absorb the $e^{\alpha r}$ growth, (\ref{afcvjfp97j8})-(\ref{cfdsjui8}) follow by the same calculations as above .
\end{proof}

And we also need the boundedness of Riesz transform on weighted $L^2$ space.
\begin{lemma}\label{degft5645}
Let $0<\alpha<\frac{1}{2}$, then $\nabla D^{-1}$ is bounded from $\rho^{-\alpha}L^2$ to $\rho^{-\alpha}L^2$.
\end{lemma}
\begin{proof}
The proof is adapted from [Theorem 6.1,\cite{Sataraia}]. For $f\in C^{\infty}_0$, we have
\begin{align*}
D^{-1}f(x)=\int_{\Bbb D} E(d(x,z))f(z)dz,
\end{align*}
where $E$ which denotes the Schwartz kernel for $D^{-1}$ is defined by:
\begin{align*}
E(t)=\int^{\infty}_0(\lambda^2+\frac{1}{4})^{-\frac{1}{2}}\varphi_{\lambda}(t)|c(\lambda)|^{-2}d\lambda,
\end{align*}
where $\phi_{\lambda}$ is the spherical function and $c(\lambda)$ is the Harish-Chandra $c$-function (see Section 2). Let $\chi(\tau)$ be a cutoff function which equals one when $|\tau|\le1$ and vanishes for $|\tau|\ge2$.
Split $\nabla D^{-1}f$ into the local and global parts:
\begin{align}
\nabla D^{-1}f&=\int_{\Bbb D}\chi(d(x,z))\nabla_xE(d(x,z))f(z)dz\label{dasaseowuiyyn}\\
&+\int_{\Bbb D}(1-\chi(d(x,z)))\nabla_xE(d(x,z))f(z)dz.\label{use3456pprrt434er}
\end{align}
The local part (\ref{dasaseowuiyyn}) is bounded on $L^p$ by [Theorem 4.7,\cite{Sataraia}]. Meanwhile, since $d(x,y)\le1$ in the local part, one has $\rho^{\alpha}(x)\rho^{-\alpha}(z)\le e$ in (\ref{dasaseowuiyyn}), and thus (\ref{dasaseowuiyyn}) is also bounded in $\rho^{-\alpha}L^2$.  It remains to consider the global part (\ref{use3456pprrt434er}).
The proof of [Proposition 4.5,\cite{SaTa}] has shown that for $0<\varepsilon \ll 1$, $(1-\chi(d(x,z)))E(d(x,z))$ satisfies
\begin{align}\label{poinfggrtj}
\left|(1-\chi(d(x,z)))E(d(x,z))\right|\lesssim e^{-(1+\varepsilon)\frac{1}{2} d(x,z)}(1+K_{\varepsilon}(d(x,z)),
\end{align}
with $\int^{\infty}_0|K_\varepsilon(r)|^2dr<\infty$.
[\cite{Sataraia}, Theorem 6.1] pointed out the same bound (\ref{poinfggrtj}) holds for $(1-\chi(d(x,z)))\nabla_xE(d(x,z))$. Thus one has
\begin{align*}
 &\left|\int_{{\Bbb D}} {{\rho ^\alpha }(x)\left( {1 -\chi(d(x,z))} \right){\nabla _x}E(d(x,z))f(z)dz}\right|\\
 &\lesssim \int_{{\Bbb D}} {{\rho ^\alpha }(x){e^{ - (1 + \varepsilon )\frac{1}{2} d(x,z)}}\left( {1 + {K_\varepsilon }(d(x,z))} \right)\left| {f(z)} \right|dz}  \\
 &\lesssim \int_{{\Bbb D}} {{\rho ^\alpha }(x){e^{ - (1 + \varepsilon )\frac{1}{2} d(x,z)}}\left( {1 + {K_\varepsilon }(d(x,z))} \right){\rho ^{ - \alpha }}(z){\rho ^\alpha }(z)\left| {f(z)} \right|dz}.
 \end{align*}
Since ${\rho ^\alpha }(x){\rho ^{-\alpha}}(z){e^{ -\alpha d(x,z)}} \le 1$, for $\varepsilon ' = \frac{1}{2} \varepsilon  - \alpha $ we have
\begin{align}
(\ref{use3456pprrt434er})\lesssim \int_{{\Bbb D}} {{e^{ - (1 + \varepsilon ')\frac{1}{2} d(x,z)}}\left( {1 + {K_\varepsilon }(d(x,z))} \right){\rho ^\alpha }(z)\left| {f(z)} \right|dz}.\label{81290pr1}
\end{align}
Applying the Kunze-Stein phenomenon (\ref{phenomenon}) and the bound $\varphi_0(r)\lesssim (1+r)e^{-\frac{1}{2} r}$, we obtain
$${\left\| (\ref{81290pr1})  \right\|_{{L^2}({\Bbb D})}}\lesssim {\left\| {{\rho ^\alpha }f} \right\|_{{L^2}({\Bbb D})}}\left\{ {\int_0^\infty  {{{\left( {\sinh r} \right)}}(1 + r){e^{ - (1 + \varepsilon ')\frac{1}{2} r}}{e^{ - \frac{1}{2} r}}dr} } \right\}.
$$
Hence, for $0<\alpha<\frac{1}{2}$, $\nabla D^{-1}$ belongs to $L(\rho^{-\alpha}L^2;\rho^{-\alpha}L^2)$.
\end{proof}

\begin{lemma}\label{eqicfder}
Let $0<\mu_1\ll1$, and $H$ satisfy the assumptions in Corollary \ref{3}.
For any $s\in[0,\frac{1}{2}]$ and any $p\in[2,\infty)$, there exists some constant $C(p,s)>0$ such that for all $f\in \mathcal{H}^{2s,p}$
\begin{align}
\frac{1}{C}\left\| H^sf \right\|_{L^p}& \le \left\|( { - \Delta })^sf \right\|_{L^p}\le C\left\| {H}^sf \right\|_{L^p}\label{pouihbhuyu}\\
\|Df-H^{\frac{1}{2}}\|_{\rho^{-\alpha}L^2}&\lesssim \mu_1\|f\|_{\rho^{-\alpha}L^2}.\label{ouihbhuyu}
\end{align}
\end{lemma}
\begin{proof}
{\bf Case 1. $0<s<\frac{1}{2}$.}
Given $p\in[2,\infty)$, fix a constant $q$ such that $2\le p<q<\infty$, $\frac{1}{p}-\frac{1}{q}<\frac{1}{200}(\frac{1}{2}-s)$ in the following proof.
And the $\omega$ defined in Lemma \ref{125dcferey} is fixed to be $\frac{1}{100}(\frac{1}{2}-s)$.
Then we see  $2s+3\omega-3<-1$.
Recall $W=V+X$ defined in Section 1. Balakrishnan's formula for non-negative operators and direct calculations (see Lemma \ref{pjkiunhyo89} in Appendix) give,
\begin{align}\label{fdffgr45678}
{(-\Delta)^s}f - {H^s}f = c(s)\int_0^\infty  {{\lambda ^s}{{(\lambda  - \Delta )}^{ - 1}}W {{(\lambda - H)}^{ - 1}}fd\lambda }.
\end{align}
Let $\lambda=\sigma^2-\frac{1}{4}$, (\ref{fdffgr45678}) yields
\begin{align}
 &{( - \Delta )^s}f - {H^s}f\nonumber \\
 &= 2c(s)\int^\infty_{\frac{1}{2}}  (\sigma ^2 -1/4)^s (- \Delta+\sigma^2-1/4 )^{ - 1}W(-H+\sigma ^2 - 1/4)^{ - 1}f\sigma d\sigma.\label{kingyui}
\end{align}
For $\sigma\ge \frac{1}{2}$, $\frac{1}{a}+\frac{1}{q}=\frac{1}{p}$, (\ref{Vab1}), (\ref{yvui}), (\ref{yvui3}) and H\"older show that
\begin{align*}
 {\left\| {{V}{{( - \Delta+{\sigma ^2} - 1/4)}^{ - 1}}} \right\|_{{L^p} \to {L^p}}} &\lesssim {\left\| {{V}} \right\|_{{L^a}}}\left\| ( - \Delta +\sigma^2 - 1/4)^{ - 1} \right\|_{{L^p} \to {L^q}}\lesssim \mu_1  \\
 \left\|X ( - \Delta+{\sigma ^2} - 1/4)^{-1}\right\|_{{L^p} \to {L^p}} &\lesssim  \left\| A \right\|_{L^{a}}\left\| \nabla( - \Delta +\sigma^2 - 1/4 )^{ - 1} \right\|_{{L^p} \to {L^q}} \lesssim  \mu_1.
\end{align*}
Hence we get ${\left\| (V+X)( - \Delta+\sigma^2-1/4)^{ - 1} \right\|_{{L^p} \to {L^p}}} \lesssim \mu_1$,
by which it follows that
\begin{align}\label{yhy7vui3}
\left\| \left( I + W( - \Delta+\sigma^2 - 1/4)^{ - 1} \right)^{ - 1} \right\|_{{L^p} \to {L^p}} \le 1.
\end{align}
Therefore, we conclude from the resolvent identity that
\begin{align}
&\left\| (H+\sigma^2-1/4 )^{-1} \right\|_{L^p \to L^q}\nonumber\\
&\lesssim  \left\| ( - \Delta+\sigma ^2 -1/4 )^{ - 1} \right\|_{L^p \to L^q}\left\| \left( I + W( - \Delta+\sigma ^2- 1/4)^{ - 1} \right)^{ - 1} \right\|_{L^p \to {L^p}}\nonumber\\
&\lesssim \min(1,{\sigma ^{ - 2+\omega}}),\label{9876gvrft}
\end{align}
Consequently, for any $2\le k< p<q<\infty$ and $\frac{1}{l}+\frac{1}{q}=\frac{1}{k}$, by H\"older we get
\begin{align}
&\left\| (- \Delta+\sigma^2-1/4 )^{ - 1}V(  H + {\sigma ^2} -1/4 )^{ - 1} \right\|_{L^p \to L^p}\nonumber\\
&\lesssim  \left\| (-\Delta+\sigma^2 - 1/4)^{ - 1} \right\|_{L^k \to L^{p}}\left\| V \right\|_{L^{l}}\left\| \left( H + {\sigma ^2} - 1/4 \right)^{ - 1} \right\|_{L^p \to L^q}\nonumber\\
&\lesssim  \mu_1 \min(1,\sigma^{-4+2\omega}).\label{yyu7vui3}
\end{align}
The left is to bound the magnetic part. By the formal resolvent identity and (\ref{yvui3}), (\ref{yhy7vui3}),
\begin{align*}
&\left\| \nabla (H+\sigma^2-1/4 )^{-1} \right\|_{L^p\to {L^q}}\\
&\lesssim \left\| \nabla(- \Delta+\sigma^2- 1/4 )^{- 1}\right\|_{L^p \to L^q}\left\| ( I + W( -\Delta+\sigma^2 -1/4)^{- 1})^{-1} \right\|_{{L^p} \to {L^p}} \\
&\lesssim \min(1,\sigma ^{ -1+2\omega}).
\end{align*}
Thus for $2\le k<p<q<\infty$ and $\frac{1}{l} + \frac{1}{q} = \frac{1}{k}$ we have
\begin{align}
 &\left\| ( - \Delta+\sigma^2 - 1/4 )^{-1}X (H + {\sigma^2} - 1/4 )^{ - 1} \right\|_{L^p\to {L^p}}\nonumber\\
 &\lesssim  \left\| ( - \Delta+\sigma ^2- 1/4)^{- 1} \right\|_{L^k\to L^p}\left\| A \right\|_{L^l}\left\| \nabla  (H + {\sigma ^2} -1/4)^{ - 1} \right\|_{{L^p} \to {L^q}}\nonumber\\
 &\lesssim  \mu_1\min (1,{\sigma ^{-3+3\omega}}).\label{yvxcd4ui3}
\end{align}
Therefore, (\ref{yvxcd4ui3}), (\ref{yyu7vui3}) and (\ref{kingyui}) yield for any $p\ge2$, there exists some $C(p)>0$ such that
\begin{align*}
\left\| (- \Delta )^sf\right\|_{{L^p}} \le C(p)\mu_1 {\left\| f \right\|_{{L^p}}} + \|{H}^sf\|_{{L^p}}.
\end{align*}
Thus by the inequality $\|f\|_{L^p} \lesssim \|D^sf\|_{{L^p}}$ for any $s>0$, we absorb the $\mu_1 \| f \|_{{L^p}}$ to the LHS by taking $\mu_1$ to be sufficiently small. Therefore, we conclude
\begin{align}\label{yui3}
\left\|( {- \Delta} )^sf \right\|_{{L^p}} \lesssim \| H^sf \|_{{L^p}}.
\end{align}
(\ref{yvxcd4ui3}), (\ref{yyu7vui3}) and (\ref{kingyui}) also yield
\begin{align*}
\|  H^sf \|_{L^p} \le C\mu_1 \| f\|_{{L^p}} +\|(- \Delta)^sf\|_{L^p}.
\end{align*}
Thus (\ref{yui3}) and the inequality $\| f \|_{{L^p}}\lesssim \| D^sf|_{L^p}$ for any $s>0$ imply
$$
{\left\| f \right\|_{{L^p}}} \lesssim {\left\| {{{H}^s}f} \right\|_{{L^p}}}.
$$
Therefore, absorbing  the term $\mu_1{\left\| f \right\|_{{L^p}}}$ to the LHS by letting $\mu_1$ to be sufficiently small gives our lemma.\\
{\bf Case 2. $s=\frac{1}{2}$.} In this case, given $q_1\in[2,\infty)$, fix a constant $p_1$ such that $2\le p_1<q_1<\infty$, $\frac{1}{p_1}-\frac{1}{q_1}<\frac{1}{200}$ in the following proof.
And the $\omega$ defined in Lemma \ref{125dcferey} is fixed to be $\frac{1}{100}$.
Instead of (\ref{fdffgr45678}), we use the following "inverse" direction identity: (see Lemma \ref{pjkiunhyo89} below)
\begin{align}\label{5fdg}
{(-\Delta)^{\frac{1}{2}}f - H^{\frac{1}{2}}}f =-c(s)\int_0^\infty \lambda ^{\frac{1}{2}}(\lambda - H)^{ - 1}W(\lambda  - \Delta )^{-1} fd\lambda .
\end{align}
Let $\sigma^2-\frac{1}{4}=\lambda$, (\ref{5fdg}) becomes
\begin{align}
&{(-\Delta)^{\frac{1}{2}}}f - {H^{\frac{1}{2}}}f \nonumber\\
&=-2c(s)\int^\infty_{\frac{1}{2}} (\sigma^2-1/4)^{\frac{1}{2}}(\sigma^2-1/4- H)^{ - 1}W(\sigma^2-1/4 - \Delta )^{-1} f\sigma d\sigma .\label{5fdg1}
\end{align}
The same arguments as proving (\ref{9876gvrft}) show
\begin{align}\label{ohnxheyeu}
\left\| (H+\sigma^2-1/4 )^{-1} \right\|_{L^{p_1} \to L^{q_1}}\lesssim \min(1,{\sigma ^{ - 2+\omega}}).
\end{align}
And for $\frac{1}{l_1}+\frac{1}{q_1}=\frac{1}{p_1}$, the same arguments give
\begin{align}
 &\left\|V  ( - \Delta+\sigma^2 - 1/4 )^{-1}\right\|_{L^{p_1}\to {L^{q_1}}}\lesssim \left\| V\right\|_{L^{l_1}}\left\| ( - \Delta+\sigma^2 - 1/4 )^{-1} \right\|_{{L^{p_1}} \to {L^{q_1}}}\nonumber\\
 &\lesssim  \mu_1\min (1,{\sigma ^{-2+\omega}}).\label{y3po9o8uhb}
\end{align}
The key point is the way of dealing with $X  ( - \Delta+\sigma^2 - 1/4 )^{-1}$ to avoid the loss of decay of $\sigma$. In fact, by the equivalence of $\|\nabla f\|_{L^p}$ and $\|D f\|_{L^p}$, we have
\begin{align*}
 \left\|X  ( - \Delta+\sigma^2 - 1/4 )^{-1}\right\|_{L^{p_1}\to {L^{q_1}}}\lesssim \|A\|_{L^{l_1}} \left\| D( - \Delta+\sigma ^2- 1/4)^{- 1} \right\|_{{L^{p_1}} \to {L^{q_1}}}.
\end{align*}
Since $D$ commutes with $D( - \Delta+\sigma ^2- 1/4)^{- 1}$, we have for any $f\in \mathcal{H}^{\frac{1}{2},p_1}$,
\begin{align}\label{pifouyuh}
\|X  ( - \Delta+\sigma^2 - 1/4 )^{-1}f\|_{{L^{q_1}}}\lesssim \mu_1\min (1,{\sigma ^{-2+\omega}})\|Df\|_{L^{p_1}}.
\end{align}
Therefore, (\ref{ohnxheyeu}), (\ref{y3po9o8uhb}), (\ref{pifouyuh}) and (\ref{5fdg1}) show
\begin{align}
\|D f - H^{\frac{1}{2}}f \|_{L^{p_1}}
\lesssim \mu_1(\|Df\|_{L^{p_1}}+\|f\|_{L^{p_1}})\int^\infty_{\frac{1}{2}} (\sigma^2-1/4)^{\frac{1}{2}}\min (1,{\sigma ^{-4+2\omega}}) \sigma d\sigma.\label{5fdghg1}
\end{align}
Thus by $\|f\|_{L^{p_1}}\lesssim \|Df\|_{L^{p_1}}$, for any $p_1\in[2,\infty)$ and  $0<\mu_1\ll 1$, there holds
\begin{align}\label{0897bcvft}
\|D f \|_{L^{p_1}}\lesssim  \|H^{\frac{1}{2}}f \|_{L^{p_1}}.
\end{align}
The inverse direction is easy by (\ref{5fdghg1}) and $\|f\|_{L^{p_1}}\lesssim \|Df\|_{L^{p_1}}$. Thus (\ref{pouihbhuyu}) has been obtained.\\
{\bf Proof of (\ref{ouihbhuyu}).} By Lemma (\ref{125dcferey}), (\ref{5fdg1}) and similar arguments as {\bf Case 2}, one has
\begin{align*}
&\|D f - H^{\frac{1}{2}}f \|_{\rho^{-\alpha}L^{2}}\\
&\lesssim (\|A\|_{L^{\infty}}\|Df\|_{\rho^{-\alpha}L^{2}}+\|V\|_{L^{\infty}}\|f\|_{\rho^{-\alpha}L^2})\int^\infty_{\frac{1}{2}} (\sigma^2-1/4)^{\frac{1}{2}}\min (1,{\sigma ^{-4}}) \sigma d\sigma.
\end{align*}
Thus (\ref{ouihbhuyu}) follows by letting $\mu_1$ be sufficiently small.
\end{proof}

In the following lemma, we prove the equivalence between $D$ and $H^{\frac{1}{2}}$ in the weighted $L^2$ space.
\begin{lemma}\label{equi2}
Let $H$ satisfy the assumptions in Corollary \ref{3}. For $0<\alpha\ll1$, $0<\mu_1\ll1$, we have
\begin{align}
\frac{1}{C}{\left\|H^{\frac{1}{2}}f \right\|_{{\rho ^{ - \alpha }}{L^2}}} &\le {\left\| Df \right\|_{{\rho ^{ - \alpha }}{L^2}}} \le C{\left\| {{{H}^{\frac{1}{2}}}f} \right\|_{{\rho ^{ - \alpha }}{L^2}}}.\label{ji876}\\
{\left\| {\nabla f} \right\|_{{\rho ^{ - \alpha }}{L^2}}} &\le C{\left\| {{{H}^{\frac{1}{2}}}f} \right\|_{{\rho ^{ - \alpha }}{L^2}}}.\label{equi3}
\end{align}
\end{lemma}
\begin{proof}
It is easy to see (\ref{equi3}) follows directly from (\ref{ji876}) and Lemma \ref{degft5645}. Thus it suffices to prove (\ref{ji876}).
By $|\nabla d(x,0)|=1$ and Leibnitz rule we have
$${\left\| {{\rho ^\alpha }f} \right\|_{{L^2}}} \lesssim  {\left\| {\nabla \left( {{\rho ^\alpha }f} \right)} \right\|_{{L^2}}} \le {\left\| {{\rho ^\alpha }\nabla f} \right\|_{{L^2}}} + C\alpha {\left\| {{\rho ^\alpha }f} \right\|_{{L^2}}}.
$$
Then the smallness of $\alpha$ implies
$${\left\| f \right\|_{{\rho ^{ - \alpha }}{L^2}}} \le C{\left\| {\nabla f} \right\|_{{\rho ^{ - \alpha }}{L^2}}}.
$$
By Lemma \ref{degft5645}, we get
\begin{align}\label{poxde45}
\left\| f \right\|_{\rho ^{- \alpha }L^2} \lesssim \|Df\|_{\rho ^{ - \alpha }{L^2}}.
\end{align}
Meanwhile, (\ref{ouihbhuyu}) gives
\begin{align}
{\left\| Df \right\|_{{\rho^{-\alpha}L^2}}} \le C\mu_1 {\left\| f \right\|_{{\rho^{-\alpha}L^2}}} + {\left\| {{{ { H}}^{\frac{1}{2}}}f} \right\|_{{\rho^{-\alpha}L^2}}}.\label{p87gvbt}
\end{align}
Since $\mu_1$ is sufficiently small, the RHS of (\ref{p87gvbt}) can be absorbed to the LHS as before.
Thus the second inequality of (\ref{ji876}) is done. The first inequality of (\ref{ji876}) follows by (\ref{ouihbhuyu}), (\ref{poxde45}) and the second one.
\end{proof}

\subsection{Conclusion}
By Lemma \ref{eqicfder}, Lemma \ref{equi2}, Lemma \ref{wqcun9}, Corollary \ref{tudiye}, we obtain Corollary \ref{3} by applying Theorem 1.1.

\section{Appendix}
\begin{lemma}\label{pjkiunhyo89}
Let $s\in(0,2)$. $H$ is the nonnegative self-adjoint operator in Corollary 1.1. The following identity holds for $f\in \mathcal{H}^2$:
\begin{align*}
H^{\frac{s}{2}}&=(-\Delta)^{\frac{s}{2}}+c(s)\int^{\infty}_0\lambda^{\frac{s}{2}}(\lambda-\Delta)^{-1}W(\lambda -H)^{-1}\mathrm{d}\lambda\\
H^{\frac{s}{2}}&=(-\Delta)^{\frac{s}{2}}-c(s)\int^{\infty}_0\lambda^{\frac{s}{2}}(\lambda -H)^{-1}W(\lambda-\Delta)^{-1}\mathrm{d}\lambda.
\end{align*}
\end{lemma}
\begin{proof}
For $s\in(0,2)$, the A.V. Balakrishnan formula for nonnegative operators $T$ is
$${T^{\frac{s}{2}}}f = c(s)\int_0^\infty  {{\tau ^{ \frac{s}{2}-1}}} {\left( {\tau  + T} \right)^{ - 1}}Tfd\tau.
$$
Then we have (\ref{pjkiunhyo89}) by direct calculations and the resolvent identity:
$$
(-\Delta-\tau)^{-1}-(-\Delta+W-\tau)^{-1}=(-\Delta-\tau)^{-1}W(-\Delta+W-\tau)^{-1}.
$$
See \cite{Laiazaea1a} for the concrete calculation.
\end{proof}

\begin{lemma}\label{phb9bv}
Assume that the time supports of $F$ and $G$ are of size $2^j$. For $(p,q)\in(2,6)$ and $\sigma_1> \frac{3}{2}(\frac{1}{2}-\frac{1}{p})$,  $\sigma> \frac{3}{2}(\frac{1}{2}-\frac{1}{q})$, one has
\begin{align*}
\left| {T_{j,\infty }^{\sigma ,{\sigma _1}}(F,G)} \right|\lesssim {2^{ - \infty j}}{\left\| F \right\|_{L_t^2L_x^{p'}}}{\left\| G \right\|_{L_t^2L_x^{q'}}}.
\end{align*}
\end{lemma}
\begin{proof}
The lemma essentially belongs to \cite{AaPa2a}. As before, the results still hold after exchanging the two operators $\widetilde{D}$ and $D$ due to the equivalence  $\|\widetilde{D}\cdot\|_{L^p}\thicksim\|{D}\cdot\|_{L^p}$ for $p\in(1,\infty)$.
By complex interpolation it suffices to consider the following three cases
\begin{align*}
&(a) \mbox{ }2=q<p<6, \Re\sigma=\frac{3}{2}(\frac{1}{2}-\frac{1}{q}),\Re\sigma_1=0; \\
&(b) \mbox{ }2=p<q<6, \Re\sigma=0, \Re\sigma_1=\frac{3}{2}(\frac{1}{2}-\frac{1}{p});\\
&(c) \mbox{ }2<p=q<6, \Re\sigma=\Re\sigma_1>\frac{3}{2}(\frac{1}{2}-\frac{1}{q}),
\end{align*}
For the case $(a)$, since $\chi_{\infty}(D)\in \mathcal{B}(L^2;L^2)$,  the inhomogeneous estimate in Lemma \ref{thing} and H\"older imply for a non-endpoint admissible pair $(m,p)$
\begin{align}
\left| T^{\sigma ,{\sigma _1}}_{j,\infty}(F,G) \right|
&\le \mathop  {\left\| {\int_{t - {2^j}\lesssim s\lesssim t - {2^{j - 1}}} {{e^{\pm iD(t-s)}D^{ - {\sigma}}}} F(s)ds} \right\|_{L^{\infty}_tL_x^2}}{\left\| G \right\|_{L_t^1L_x^2}}\nonumber \\
&\lesssim {\left\| F \right\|_{L_t^{m'}L_x^{p'}}}{\left\| G \right\|_{L_t^1L_x^2}} \nonumber\\
&\lesssim {\left\| F \right\|_{L_t^2L_x^{p'}}}{\left\| G \right\|_{L_t^2L_x^2}}{2^{\frac{j}{2}}}{2^{j(\frac{1}{2} - \frac{1}{m'})}}.\label{hgt555655}
\end{align}
In the case $(b)$, the same arguments as $(a)$ yield
\begin{align}\label{vffffggggh}
\left| {T^{\sigma ,{\sigma _1}}_{j,\infty}(F,G)} \right|\lesssim {\left\| F \right\|_{L_t^2L_x^{p'}}}{\left\| G \right\|_{L_t^2L_x^2}}2^{\frac{j}{2}}2^{j(\frac{1}{2} - \frac{1}{m'})}.
\end{align}
For the case $(c)$, by Corollary \ref{hasonggu}, we have for $p\in(2,6)$ and $q\in(2,6)$
\begin{align*}
\left| {T^{\sigma ,{\sigma _1}}_{j,\infty}(F,G)} \right| &\lesssim  \mathop {\sup }\limits_t {\left\| {\int_{t - {2^j}\lesssim s\lesssim t - {2^{j - 1}}} {\chi_{\infty}(D){\widetilde{D}^{ - {\sigma _1}}D^{ - \sigma }}F(s)} ds} \right\|_{L_x^q}}{\left\| G \right\|_{L_t^1L_x^{q'}}} \\
&\lesssim {2^{ - j\infty }}\mathop {\sup }\limits_t\left({\int _{t - {2^j}\lesssim s\lesssim t - {2^{j - 1}}}}{\left\| {F(s)} \right\|_{L_x^{p'}}}ds\right
){\left\| G \right\|_{L_t^1L_x^{q'}}} \\
&\lesssim {2^{ - j\infty }}{\left\| {F(s)} \right\|_{L_t^1L_x^{p'}}}{\left\| G \right\|_{L_t^1L_x^{q'}}}.
\end{align*}
Hence H\"older gives
\begin{align}\label{jnbbvvffffggggh}
\left| {T^{\sigma ,{\sigma _1}}_{j,\infty}(F,G)} \right|\lesssim {2^{ - j\infty }}{\left\| {F(s)} \right\|_{L_t^2L_x^{p'}}}{\left\| G \right\|_{L_t^2L_x^{q'}}}.
\end{align}
Interpolating (\ref{hgt555655}), (\ref{vffffggggh}), (\ref{jnbbvvffffggggh}) gives our lemma.
\end{proof}

\noindent Ze Li. \textit{rikudosennin@163.com}\\
\noindent{\small{Academy of Mathematics and Systems Science (AMSS)}}\\
\noindent{\small{Chinese Academy of Sciences (CAS)}}\\
\noindent{\small{Beijing 100190, P. R. China}}\\

\end{document}